%% file: main.tex
\documentclass[11pt]{article}

\input{custom_stuff}

\title{Ramsey numbers for multiple copies of sparse graphs}
\author{
  Aurelio Sulser\thanks{Department of Mathematics, ETH Z\"{u}rich, 8092
  Z\"{u}rich, Switzerland. Email: \texttt{asulser@student.ethz.ch}}
  \and
  Milo\v{s} Truji\'{c}\thanks{Institute of Theoretical Computer Science, ETH
  Z\"{u}rich, 8092 Z\"{u}rich, Switzerland. Email: \texttt{mtrujic@inf.ethz.ch}.
  Research supported by grant no.\ 200020 197138 of the Swiss National Science
  Foundation.}
}
\date{}

\begin{document}
\maketitle

\begin{abstract}
  For a graph $H$ and an integer $n$, we let $nH$ denote the disjoint union of
  $n$ copies of $H$. In 1975, Burr, Erd\H{o}s, and Spencer initiated the study
  of Ramsey numbers for $nH$, one of few instances for which Ramsey numbers are
  now known precisely. They showed that there is a constant $c = c(H)$ such that
  $r(nH) = (2|H| - \alpha(H))n + c$, provided $n$ is sufficiently large.
  Subsequently, Burr gave an implicit way of computing $c$ and noted that this
  long term behaviour occurs when $n$ is triply exponential in $|H|$. Very
  recently, Buci\'{c} and Sudakov revived the problem and established an
  essentially tight bound on $n$ by showing $r(nH)$ follows this behaviour
  already when the number of copies is just a single exponential. We provide
  significantly stronger bounds on $n$ in case $H$ is a sparse graph, most
  notably of bounded maximum degree. These are relatable to the current state of
  the art bounds on $r(H)$ and (in a way) tight. Our methods rely on a beautiful
  classic proof of Graham, R\"{o}dl, and Ruci\'{n}ski, with the emphasis on
  developing an efficient absorbing method for bounded degree graphs.
\end{abstract}

\input{introduction}

\input{preliminaries}

\input{absorbers}

\input{asymmetric_case}

\input{symmetric_case}

\input{concluding}

{\small \bibliographystyle{abbrv} \bibliography{references}}

\end{document}

%% file: custom_stuff.tex
\usepackage[utf8]{inputenc}
\usepackage[british]{babel}
\usepackage[a4paper, margin=2.5cm]{geometry}
\usepackage{amsmath, amsthm, amssymb, amsfonts}
\usepackage{mathtools}
\usepackage{thmtools}
\usepackage[shortlabels, inline]{enumitem}
\usepackage[font={small, it}]{caption}
\usepackage{subcaption}
\usepackage{microtype}
\usepackage{xcolor}
\usepackage{mathabx}
\usepackage{hyperref}
\usepackage{bm}
\usepackage{tikz}

\newcommand{\cD}{\ensuremath{\mathcal D}}

\newcommand{\cG}{\ensuremath{\mathcal G}}
\newcommand{\cH}{\ensuremath{\mathcal H}}

\newcommand{\eps}{\varepsilon}
\renewcommand{\phi}{\varphi}
\renewcommand{\rho}{\varrho}

\DeclareMathOperator*{\N}{\mathbb{N}}

\DeclarePairedDelimiter\floor{\lfloor}{\rfloor}

\let\setminus=\smallsetminus

\declaretheorem[parent=section]{theorem}
\declaretheorem[sibling=theorem]{lemma}
\declaretheorem[sibling=theorem]{proposition}
\declaretheorem[sibling=theorem]{claim}

\setlist{itemsep=0.1em, topsep=0.1em, parsep=0.1em, partopsep=0.1em}

\hypersetup{
  colorlinks,
  linkcolor={red!60!black},
  citecolor={green!50!black},
  urlcolor={blue!80!black}
}

\colorlet{RoyalRed}{red!70!black}
\definecolor{RoyalBlue}{rgb}{0.25, 0.41, 0.88}
\definecolor{RoyalAzure}{rgb}{0.0, 0.22, 0.66}

\newlength{\bibitemsep}
\newlength{\bibparskip}
\let\oldthebibliography\thebibliography
\renewcommand\thebibliography[1]{%
  \oldthebibliography{#1}%
  \setlength{\parskip}{\bibitemsep}%
  \setlength{\itemsep}{\bibparskip}%
}

\newcounter{propcnt} % counter

\newlist{alphenum}{enumerate}{1}
\setlist[alphenum,1]{%
  label=\normalfont{\bfseries{(\Alph{propcnt}{\arabic*})}},
  ref=\normalfont{(\Alph{propcnt}{\arabic*})},
  leftmargin = \parindent+3em,
}

%% file: introduction.tex
\section{Introduction}

Ramsey theory is a big and important area of mathematics that revolves around a
simple paradigm: `In every large chaotic system there has to exist a
well-organised subsystem'. Here we are concerned with graph Ramsey theory. A
classical result in this area, from which the whole theory derives its name, is
a theorem of Ramsey~\cite{ramsey1929problem} from 1929. It states that for every
graph $H$ there is a large integer $N$ such that in every $2$-colouring of the
edges of the complete graph on $N$ vertices $K_N$ there is a monochromatic copy
of $H$. The smallest $N$ with this property is called the \emph{Ramsey number}
of $H$ and denoted by $r(H)$.

The original question was mostly focused on determining Ramsey numbers for
complete graphs $K_k$. Despite much effort (see, e.g.~\cite{conlon2009new,
erdos1947some, erdos1935combinatorial, sah2020diagonal, spencer1975ramsey}, as
well as a survey~\cite{conlon2015recent}) the standard bounds $2^{k/2} \leq
r(K_k) \leq 2^{2k}$ remain unchanged at large. It is not of much surprise that
as an effect of this the theory branched out attempting to determine $r(H)$ for
general classes of graphs $H$. One such direction of fundamental interest for
its own right is that of \emph{sparse graphs}. Arguably, the most natural way of
imposing sparsity is to restrict the degrees of vertices. Addressing precisely
that, Burr and Erd\H{o}s~\cite{burr1975magnitude} asked in 1975 whether every
$k$-vertex graph $H$ with maximum degree $\Delta$ satisfies $r(H) \leq
c(\Delta)k$, for some constant $c(\Delta)$ that depends only on $\Delta$. In one
of the earliest applications of the celebrated Szemer\'{e}di's regularity lemma,
Chvat\'{a}l, R\"{o}dl, Szemer\'{e}di, and Trotter~\cite{chvatal1983ramsey}
famously showed this to indeed be the case. However, due to the techniques used
the constant $c(\Delta)$ is inevitably extremely large
(see~\cite{gowers1997lower}). (We come back to this issue later.) The current
state of the art result in the area is by Conlon, Fox, and
Sudakov~\cite{conlon2012two} who showed $r(H) \leq 2^{O(\Delta\log\Delta)}k$,
which is only by the logarithmic factor in the exponent away from the lower
bound of $2^{\Omega(\Delta)}k$ (see \cite{graham2000graphs}).

In this paper we dive into a related problem of determining Ramsey numbers for
the disjoint union of $n$ copies of a graph $H$, denoted by $nH$. An intriguing
fact about the class $nH$ is that it is one of the rare instances for which the
Ramsey numbers are known to a significant extent. Studying $r(nH)$ was initiated
by Burr, Erd\H{o}s, and Spencer~\cite{burr1975ramsey} in 1975 who proved that
for any $k$-vertex graph $H$ without isolated vertices, there is a $c = c(H)$
such that if $n$ is sufficiently large, then
\begin{equation}\label{eq:long-term-behaviour}
  r(nH) = (2k - \alpha(H))n + c,
  \tag{$\star$}
\end{equation}
where $\alpha(H)$ stands for the order of the largest independent set of $H$.
To illustrate all this with a concrete example, they showed $r(nK_3) = 5n$, and
so $c(K_3) = 0$ and one can even take $n \geq 2$. The general result came with
two deficiencies pointed out by the authors themselves: it provides no way of
computing $c(H)$ nor does it quantify how large $n$ has to be for the long term
behaviour \eqref{eq:long-term-behaviour} to hold, other than it being finite.

An answer to both of these came from Burr~\cite{burr1987ramsey} who gave an
effective, yet cumbersome, way of computing $c(H)$ and, even though this was not
the focus of his research, established that the bound above starts to hold when
$n$ is a triply exponential function in $k$. In a way this gives an
unsatisfactory answer, as it may well be that the behaviour in
\eqref{eq:long-term-behaviour} starts much earlier---already when $n =
O(r(H)/|H|)$, which would be optimal. A big advancement in this direction was
recently made by Buci\'{c} and Sudakov~\cite{bucic2021tight}, showing that one
can take $n$ to be as small as $2^{O(k)}$, which is essentially tight (up to the
constant in the exponent), e.g.\ when $H = K_k$.

Our main result combines the two areas of research discussed above and provides
a leap towards the optimal bound on $n$ when $H$ is a $k$-vertex graph with
maximum degree $\Delta$.

\begin{theorem}\label{thm:main-symmetric}
  Let $H$ be a $k$-vertex graph with no isolated vertices and maximum degree
  $\Delta$. There is a constant $c = c(H)$ and $n_0 =
  2^{O(\Delta\log^2\Delta)}k$ such that for all $n \geq n_0$
  \[
    r(nH) = (2k - \alpha(H))n + c.
  \]
\end{theorem}

The requirement on $n$ differs from the best known upper bound on Ramsey numbers
for bounded degree graphs by only a $\log\Delta$ factor in the exponent, and is
thus (conditionally) almost optimal. Perhaps surprisingly, the linear dependence
on $k$ is also necessary for the long term behaviour
\eqref{eq:long-term-behaviour} to settle. We discuss this further in
Section~\ref{sec:symmetric-problem} (see
Proposition~\ref{prop:symmetric-lower-bound}).

Most of the novel ideas go into establishing a linear dependence on $k$ and
revolve around constructing efficient absorbers (see
Lemma~\ref{lem:efficient-absorbers} below). Our method further relies on a
beautiful proof of Graham, R\"{o}dl, and Ruci\'{n}ski~\cite{graham2000graphs},
who developed a tool for substituting Szemer\'{e}di's regularity lemma from
\cite{chvatal1983ramsey} and drastically reduced the constant $c(\Delta)$
mentioned before to obtain $r(H) = 2^{O(\Delta\log^2\Delta)}k$. In fact, these
density related tools allowed the same group of authors to remove one
$\log\Delta$ factor in case $H$ is a bipartite graph \cite{graham2001bipartite}.
We get a similar improvement on Theorem~\ref{thm:main-symmetric} in case $H$ is
bipartite, shown in the final section of this paper,
Section~\ref{sec:concluding-remarks}.

\paragraph{Sparsity in terms of the number of edges.} Of course, an equally
natural measure of sparsity of a graph is simply the count of its edges.
Inspired by the seemingly unapproachable question (to date there is no progress
on it) of Erd\H{o}s and Graham~\cite{erdos1973partition} on whether among all
graphs with $m = \binom{k}{2}$ edges the complete graph has the largest Ramsey
number, Erd\H{o}s~\cite{erdos1984some} conjectured (see also
\cite{chung1998erdos}) that the Ramsey number of every graph with $m$ edges is
$2^{O(\sqrt{m})}$. This appeared sensible as the number of vertices in a
complete graph with $m$ edges is a constant multiple of $\sqrt{m}$, and at the
same time would thus be optimal. After initial progress of Alon, Krivelevich,
and Sudakov~\cite{alon2003turan}, Sudakov~\cite{sudakov2011conjecture} resolved
Erd\H{o}s' conjecture in the positive.

We show a similar result to that of Theorem~\ref{thm:main-symmetric} for when
$H$ is a graph with $m$ edges.

\begin{theorem}\label{thm:main-symmetric-edges}
  Let $H$ be a graph with $m$ edges and no isolated vertices. There is a
  constant $c = c(H)$ and $n_0 = 2^{O(\sqrt{m}\log^2 m)}$ such that for all $n
  \geq n_0$
  \[
    r(nH) = (2k - \alpha(H))n + c.
  \]
\end{theorem}

To put things into context, when $m$ is roughly of the order $k^2$, then
\cite{bucic2021tight} gives an essentially optimal result. However, when $m \ll
k^2/\log^2 k = k^{2-o(1)}$, in other words when $H$ is sparse, then
$2^{O(\sqrt{m}\log^2 m)}$ is significantly better. Here again, our result is
optimal up to the logarithmic term in the exponent.

\paragraph{The asymmetric problem.} Another area in which our methods as an
aside yield an improvement is that of \emph{asymmetric} Ramsey numbers of an
arbitrary graph $G$ versus $nH$. Generalising the usual Ramsey number, we let
$r(G, H)$ stand for the smallest $N$ such that every $2$-colouring of the edges
of $K_N$ contains either $G$ in the first colour or $H$ in the second. Burr
studied $r(G, nH)$ for its independent interest, but also as an important piece
of determining $r(nH)$. Similarly as in the symmetric case, Buci\'{c} and
Sudakov obtained an exponential improvement over the result of
Burr~\cite{burr1987ramsey} showing that the long term behaviour of $r(G, nH)$
starts already when $n \geq 2^{O(k)}$, where $k = \max\{|G|,|H|\}$.

The asymmetric problem is significantly easier and is a good model for testing
out ideas and developing intuition for the symmetric one. That said, we provide
two improvements over the best known results in this scenario, which, with
almost no additional effort, come out of the techniques used to tackle the
symmetric case. We address this in Section~\ref{sec:asymmetric-problem} in hope
it `eases' the reader into a much more involved argument to follow in
Section~\ref{sec:symmetric-problem}.

In order to state our results, we need to introduce a bit of notation. For two
families of graphs $\cG$, $\cH$, we denote by $r(\cG, \cH)$ the smallest $N$
such that every 2-colouring of the edges of $K_N$ contains a copy of some $G \in
\cG$ in the first colour or a copy of some $H \in \cH$ in the second colour.
For a graph $G$, $\cD(G)$ stands for the class of graphs obtained by removing a
maximal independent set of $G$.

\begin{theorem}
  \label{thm:main-asymmetric}
  Let $G$ be a non-empty and $H$ a connected\footnote{The fact that $H$ is
    connected is necessary, but a similar result can be obtained in case of
  disconnected graphs.} graph with $k = \max\{|G|,|H|\}$. Then, provided $n =
  O(k^{10}r(G,H)^2)$, we have
  \[
    r(G,nH) = n|H| + r(\cD(G),H) - 1.
  \]
\end{theorem}

This produces a notable improvement over the previously known bound on $n$ in
all cases for which $r(G,H)$ is not exponential in $k$, and is almost optimal
with respect to dependency on $r(G,H)$. If, additionally, $G$ and $H$ are (very)
sparse, in terms of maximum degree, we can do even better.

\begin{theorem}
  \label{thm:main-asymmteric-bounded-deg}
  Let $G$ be a non-empty and $H$ a connected graph, both of maximum degree
  $\Delta$ and with $k = \max\{|G|,|H|\}$. Then, provided $n =
  2^{O(\Delta\log^2\Delta)}k/|H|$, we have
  \[
    r(G,nH) = n|H| + r(\cD(G),H) - 1.
  \]
\end{theorem}

Note that for bounded degree graphs, $2^{O(\Delta)}k \leq r(G, H) \leq
2^{O(\Delta\log\Delta)}k$. Therefore, the comparison between the two is a bit
intricate, with the latter giving better bounds when the graphs are very sparse,
e.g.\ when $\Delta\log^2\Delta \ll \log k$. Interestingly, the requirement on
$n$ for the long term behaviour in it does not even depend on $k$ when $G$ and
$H$ are of roughly the same order, bringing us closer to the ideal scenario of
$n = O(r(G,H)/|H|)$.

%% file: preliminaries.tex
\section{Preliminaries}

For a graph $G$, we let $|G|$ denote its order, i.e.\ number of vertices. Given
a set $X \subseteq V(G)$, we write $G - X$ for $G[V(G) \setminus X]$. To reduce
notational overhead, we often use a graph interchangeably with its vertex set;
in particular, we write $G$ for $V(G)$ in some set-theoretic notation. The
\emph{density} of a set $X$ is $e(X)/\binom{|X|}{2}$, where $e(X)$ stands for
the number of edges contained in $X$, and for two disjoint vertex sets $X$ and
$Y$, we let density be $d(X, Y) = e(X, Y)/(|X||Y|)$, where $e(X, Y)$ is the
number of edges with one endpoint in $X$ and the other in $Y$. A set $U$ is said
to be $(\eps,\gamma)$-dense in a graph $G$ if every $X \subseteq U$, with $|X|
\geq \eps|U|$, has density at least $\gamma$. Similarly, $U$ is said to be
\emph{bi-$(\eps,\gamma)$-dense}, if for all disjoint $X, Y \subseteq U$ of order
$|X|, |Y| \geq \eps|U|$ the density between them satisfies $d(X, Y) \geq
\gamma$. A graph $G$ is bi-$(\eps,\gamma)$-dense if $V(G)$ is.

Throughout, we think of $H$ as a graph on $k \in \N$ vertices on the vertex set
$[k] = \{1,\dotsc,k\}$. If $H$ has maximum degree $\Delta$, then the order of
its largest independent set $\alpha := \alpha(H)$ satisfies $\frac{k}{\Delta+1}
\leq \alpha \leq (1-\frac{1}{\Delta+1})k$. Unless specified differently, we use
$I$ to denote a largest independent set of $H$.

The logarithm function is always in base two, and we omit floors and ceilings
whenever they are not of vital importance. In an attempt to increase
readability, we often abuse the asymptotic notation $f = o(g)$. As an example,
if $g(k) = 2^{O(\Delta\log^2\Delta)}k$, and we write $f = o(g)$, for $f(k) =
2^{c\Delta\log\Delta}k$, we implicitly mean that there exists a choice of a
constant in $2^{O(\cdot)}$ for $g(k)$ (depending on $c$) so that $f$ is much
smaller than $g$. We stress this out several times when it appears in the paper.

In a tight relation to this, we extensively use an upper bound on the Ramsey
number $r(H)$ when $H$ is a $k$-vertex graph with maximum degree $\Delta$:
\begin{equation}\label{eq:ramsey-bound}
  r(H) \leq 2^{84\Delta+2} \Delta^{32\Delta} k \leq 2^{128\Delta\log\Delta} k.
\end{equation}
This comes from \cite[Theorem~2.1]{conlon2012two}. We apply the same upper bound
on the asymmetric $r(G,H)$ when both $G$ and $H$ are graphs of maximum degree
$\Delta$ and $k = \max\{|G|,|H|\}$. (It may be worth pointing out that we can
work with a weaker bound of Graham, R\"{o}dl, and Ruci\'{n}ski as well.)
In our proofs, we avoid being strict with computing the right constant(s) in the
exponent, and instead, for the sake of clarity, prefer loosely adding them up
mainly being driven by the one outlined here.

A key lemma used as a substitute for the Szemeredi's regularity lemma in
\cite{graham2000graphs}, shows that one can go from a dense graph to a (much!)
smaller bi-dense graph with a loss in terms of parameters.

\begin{lemma}[Lemma~1 in \cite{graham2000graphs}]\label{lem:dense-to-bi-dense}
  For all $s, \beta, \eps, \gamma$ such that $0<\beta, \gamma, \eps < 1$, $s
  \geq \log(4/\gamma)$ and $(1-\beta)^{2s} \geq 2/3$, the following holds.
  If $G$ is a $((2\eps)^s\beta^{s-1},\gamma)$-dense graph on $N$ vertices, then
  there exists $U \subseteq V(G)$ with $|U| = \eps^{s-1}\beta^{s-2} N$ such that
  $G[U]$ is bi-$(\eps,\gamma/2)$-dense.
\end{lemma}

\subsection{Embedding bounded degree graphs}

The results mentioned (or proved) here follow the standard way of embedding a
bounded degree graph $H$ into a larger graph $G$: pick an ordering of the
vertices of $H$ and embed them one by one. Of course, some requirements on $G$
need to be imposed for this to make sense. As a light warm-up, we begin with a
lemma whose proof is a simple exercise in graph theory, and is thus omitted.

\begin{lemma}\label{lem:greedy-embedding}
  Let $H$ be a $k$-vertex graph with maximum degree $\Delta$ and $G$ a graph on
  $N \geq 4k$ vertices with density at least $1-1/(8\Delta)$. Then $G$ contains
  $H$ as a subgraph.
\end{lemma}

The next lemma is our essential embedding tool, which is mainly standard
nowadays, with one addition. We would want to find an embedding of $H$ into $G$
such that every vertex in it acting as an image of some $i \in I$ can be
substituted by plenty other vertices to again form a copy of $H$. A bit more
formally, for a copy $H'$ of $H$ in $G$, we say that a vertex $u \in G \setminus
H'$ is an \emph{alias} of some $v \in H'$ if $u$ together with $H' - v$ forms a
copy of $H$ in $G$ as well (see Figure~\ref{fig:absorbers} below where a similar
concept is used).

\begin{lemma}\label{lem:embedding-lemma}
  Let $H$ be a graph on $k$ vertices with maximum degree $\Delta$. Let $G$ be a
  bi-$(\eps,2\gamma)$-dense graph and $V_i \subseteq V(G)$, $i \in [k]$, a
  collection of (not necessarily disjoint) sets such that
  \[
    \gamma^{\Delta-r}|V_i| - 2r\eps|G| \geq k \qquad \text{for all $r \in
    \{0,\dotsc,\Delta\}$ and $i \in \{1,\dotsc,k\}$}.
  \]
  Given any independent set $I$ of $H$, there exists an embedding of $H$ into
  $G$ such that the image of every $i \in I$ has at least $\gamma^\Delta|V_i|-k$
  aliases.
\end{lemma}

\begin{proof}
  Order the vertices $\{1,\dotsc,k\}$ of $H$ by putting the ones in $I$ last.
  Let $S_{i-1} := \{1,\dotsc,i-1\}$ for all $i \in [k]$ and let $N_H(i,S_{i-1})$
  stand for the neighbourhood of $i$ intersecting $S_{i-1}$ in $H$. We embed the
  vertices of $H$ one by one, each $i$ into its corresponding $V_i$ by
  maintaining the following property for all $j > i$: the candidate set $C_j^i$,
  that is the common neighbourhood of image of $N_H(j,S_i)$ in $V_j$, is of
  order at least $\gamma^{|N_H(j,S_i)|}|V_j|$. The fact that this can be done is
  easily verified by induction. For $i = 0$ there is nothing to prove. Assume
  we are embedding $i \geq 1$ next and let $J \subseteq [k]$ be its $r \leq
  \Delta$ not-yet-embedded neighbours in $H$ (if $J$ is empty there is nothing
  to worry about). We need to find a vertex $v$ among the previously unused
  vertices of $C_i^{i-1}$ such that for all $j \in J$, $v$ has at least
  $\gamma|C_j^{i-1}|$ neighbours in $G$ among the candidates for $j$. Note that
  each $C_j^{i-1}$ by induction hypothesis and assumption of the lemma satisfies
  \[
    |C_j^{i-1}| \geq \gamma^{\Delta-1} |V_j| \geq 2\eps|G|.
  \]
  Consider a set $X \subseteq V(G)$ of vertices with fewer than
  $\gamma|C_j^{i-1}|$ neighbours in $C_j^{i-1}$ and assume for contradiction
  that $|X| > 2\eps|G|$. Then $X \cap C_j^{i-1}$ can be partitioned into $Q_1
  \cup Q_2$ such that $\tilde X = X \setminus Q_1$ and $\tilde C_j^{i-1} =
  C_j^{i-1} \setminus Q_2$ are disjoint and both of order at least $\eps|G|$.
  But then the number of edges in $G[\tilde X,\tilde C_j^{i-1}]$ is less than
  $\gamma|\tilde X||C_j^{i-1}| \leq 2\gamma|\tilde C_j^{i-1}||\tilde X|$, which
  is a contradiction with $G$ being bi-$(\eps,2\gamma)$-dense. In conclusion,
  there are at most $2\eps|G|$ vertices in $G$ with fewer than $\gamma|C_j^{i-1}|$
  neighbours in $C_j^{i-1}$. Again by the induction hypothesis and the
  assumption of the lemma, we have that there are at least
  \[
    |C_i^{i-1}| - (k-1) - 2r\eps|G| \geq \gamma^{\Delta-r}|V_i| - (k-1) -
    2r\eps|G| \geq 1,
  \]
  available good choices for mapping $i$ into. Lastly, for any vertex $i \in I$,
  right after the time of its embedding the set of aliases of its image is of
  order at least $\gamma^{\Delta}|V_i| - k$ and does not shrink by embedding any
  $j \in I$ with $j > i$, as $I$ is an independent set.
\end{proof}

A major downside of using the Lemma~\ref{lem:embedding-lemma} is that typically
we are forced to have $\gamma = \Theta(1/\Delta)$ which then requires the order
of the graph $G$ above to be at least $2^{O(\Delta\log\Delta)}$. We briefly
touch upon this in Section~\ref{sec:concluding-remarks}.

%% file: absorbers.tex
\section{The absorbing method}

One cannot overstate how versatile and useful this, in its core simple, method
has recently proven to be in extremal combinatorics and beyond. A list of major
problems tackled by it goes way outside the scope of this work (see,
e.g.~\cite{glock2016existence, kang2021proof, kuhn2013hamilton, kwan2022high,
montgomery2019spanning, montgomery2021proof}). Even though it may not be
immediately obvious how absorbers can be helpful in the current Ramsey setting,
Buci\'{c} and Sudakov set up a scheme that makes use of a `baby version' for
$H$-tiling\footnote{An $H$-tiling of a graph $G$ is a vertex-disjoint collection
of copies of $H$ covering all but at most $|G| \bmod |H|$ vertices of $G$. A
perfect $H$-tiling covers all vertices of $G$.} dense graphs which turned out to
be essential for their results. The most vital part of our strategy is
developing more efficient absorbers, specifically for when $H$ is a bounded
degree graph.

Let $k \in \N$ and let $H$ be a graph on $k$ vertices. For $r \in \N$, an
$(H,r)$-\emph{absorber} for a set $U$ is a graph $A$ with the following
property: for every $R \subseteq U$ of order at most $r$ there is an $H$-tiling
in $A[V(A) \cup R]$.

\subsection{General absorbers}

We first show how to find absorbers in the general setting without any maximum
degree assumptions on $H$. To do so we usually rely on smaller `local'
structures. For a set $X = \{x_1,\dotsc,x_k\}$, we say that a graph $L_X$ is a
\emph{local absorber} if both $L_X$ and $L_X \cup X$ have a perfect $H$-tiling.
A very natural (and easy) way of constructing a local absorber for a set $X$ is
to just combine a copy of $H$ with several \emph{switchers}. A graph $S$ with
two specified vertices $u$ and $v$ is a $uv$-\emph{switcher} (with respect to
$H$), if both $S-u$ and $S-v$ have a perfect $H$-tiling (see
Figure~\ref{fig:absorbers}); note that $u$ and $v$ are aliases of each other.

\begin{figure}[!htbp]
  \captionsetup[subfigure]{textfont=scriptsize}
  \centering
  \begin{subfigure}{.19\textwidth}
    \centering
    \includegraphics[scale=1]{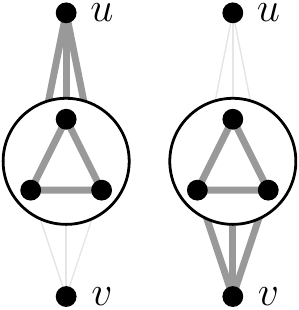}
    \caption{$uv$-switcher}
  \end{subfigure}%
  \hspace{1em}
  \begin{subfigure}{.37\textwidth}
    \centering
    \includegraphics[scale=1]{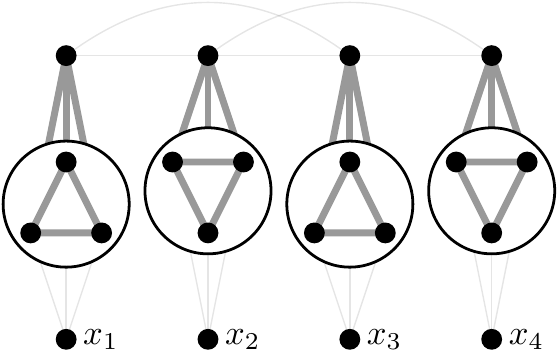}
    \caption{Local $K_4$-absorber: tiling $L_R$}
  \end{subfigure}%
  \hspace{0.5em}
  \begin{subfigure}{.37\textwidth}
    \centering
    \includegraphics[scale=1]{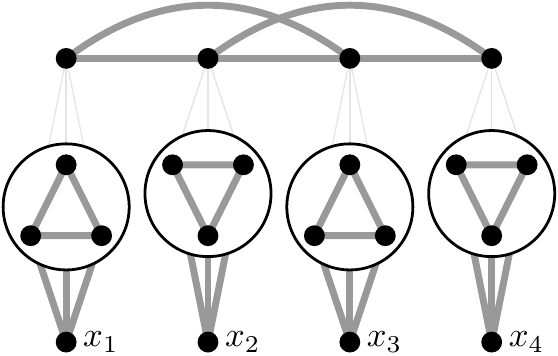}
    \caption{Local $K_4$-absorber: tiling $L_R \cup X$}
  \end{subfigure}%
  \caption{An example of a $uv$-switcher for $K_4$ (left) and its local
  absorber (middle and right)}
  \label{fig:absorbers}
\end{figure}

A much harder task is stitching local absorbers into an $(H,r)$-absorber for
some set $U$. However, the following result of
Montgomery~\cite{montgomery2019spanning} is a widely used guiding hand for
precisely that.

\begin{proposition}\label{prop:resilient-template}
  There is an integer $\ell_0 \in \N$, such that for every $\ell \geq \ell_0$,
  there exists a bipartite graph with maximum degree at most $40$ and vertex
  classes $X \cup Y$ and $Z$, with $|X| = |Y| = 2\ell$, and $|Z| = 3\ell$, such
  that for every $X' \subseteq X$ with $|X'| = \ell$, there is a perfect
  matching between $X' \cup Y$ and $Z$.
\end{proposition}

The construction of our absorbers mainly relies on the celebrated dependent
random choice~\cite{fox2011dependent}, which enables us to find for any $u, v$ a
$uv$-switcher using their common neighbourhood.

\begin{lemma}\label{lem:dependent-random-choice}
  Let $a, d, m, n, r$ be positive integers. Let $G=(V, E)$ be a graph with
  $|V|=n$ vertices and average degree $d=2|E(G)| / n$. If there is a positive
  integer $t$ such that
  \[
    \frac{d^t}{n^{t-1}} - \binom{n}{r} \Big(\frac{m}{n}\Big)^t \geq a
  \]
  then $G$ contains a subset $U$ of at least a vertices such that every $r$
  vertices in $U$ have at least $m$ common neighbours.
\end{lemma}

With the following lemma at hand the proof of
Theorem~\ref{thm:asymmetric-general-ub} becomes all but a formality. In general,
the way we apply the absorbing lemma(s) is to consider a $2$-coloured complete
graph that has no copy of $G$ in, say, blue, which then implies an absorber for
$H$ in red.

\begin{lemma}\label{lem:absorbing-lemma-general}
  Let $G$ and $H$ be two graphs and $k = \max\{|G|, |H|\}$. Let $K$ be a graph
  of order $N \geq 2^{32}k^{10}r(G, H)^2$ and whose complement is $G$-free.
  Then, there exists $U \subseteq V(K)$ of order $N/(2k^3)$ and an $(H,
  r(G,H))$-absorber for $U$ in $K$ of order at most $\sqrt{N}/k$.
\end{lemma}
\begin{proof}
  As the complement of $K$ is $G$-free, by Tur\'{a}n's theorem $e(K) \geq
  N^2/(2k)$, and so the average degree in $K$ is $d \geq N/k$. Therefore,
  setting $m = \sqrt{N}/k$,
  \[
    \frac{d^2}{N} - \binom{N}{2} \Big(\frac{m}{N}\Big)^2 \geq \frac{N}{k^2} -
    \frac{N}{2k^2} = \frac{N}{2k^2}.
  \]
  Applying Lemma~\ref{lem:dependent-random-choice} with $a = N/(2k^2)$, $r = 2$,
  and $t = 2$ to $K$, returns a set $V \subseteq V(K)$ of order $N/(2k^2)$ in
  which every two vertices have at least $m$ common neighbours; note $m \geq k^4
  r(G,H)$.

  As long as there is a vertex $v \in V$ with degree at most $|V|/k - 1$ in $V$
  we mark it and delete it together with its neighbourhood from $V$ (with a
  slight abuse of notation we still call the remaining set $V$). This process
  has to stop after at most $k-1$ iterations, as otherwise the marked vertices
  form an independent set of order $k$---a contradiction with the complement of
  $K$ being $G$-free. Therefore, we construct a set $U_0$ of order $|U_0| \geq
  |V|/k$ with minimum degree at least $|U_0|/k$. It is not too difficult to see
  that there is $U \subseteq U_0$ of order exactly $|V|/k \geq N/(2k^3)$ with
  minimum degree at least $|U|/(2k) \geq N/(4k^4)$. As the expected degree of
  any $v \in U_0$ into a uniformly random $U \subseteq U_0$ of precisely that
  size is at least $|U|/k \geq |V|/k^2$, Chernoff's inequality for
  hypergeometric distributed random variables (see,
  e.g.~\cite[Theorem~2.10]{janson2011random}) and the union bound over all
  vertices in $U_0$, show that with positive probability $U$ satisfies the
  desired property.

  Let $p = 1/(2^7\sqrt{N})$ and let $X \subseteq U$ be a set of order $2\ell =
  2|U|p$, such that every $v \in U$ has degree into $X$ at least $|U|p/k \geq
  \sqrt{N}/(2^9k^4)$. By choosing $X$ uniformly at random, an almost analogous
  application of Chernoff's inequality and the union bound over all vertices in
  $U$, show that with positive probability $X$ satisfies the desired property.

  Having these preparatory steps done, we construct an $(H,r(G,H))$-absorber $A$
  for $U$. We do so iteratively, while keeping its order at most $m/2$
  throughout the process. Let $Y$ and $Z$ be disjoint sets belonging to $V
  \setminus X$, where $|Y| = 2\ell$ and $|Z| = 3\ell(|H|-1)$, and add $X$, $Y$,
  and $Z$ to $A$. Arbitrarily partition $Z$ into $Z_1, \dotsc, Z_{3\ell}$ of
  equal order, and consider the template graph $T$ given by
  Proposition~\ref{prop:resilient-template} with vertex classes $X \cup Y$ and
  $\{Z_1,\dotsc,Z_{3\ell}\}$. For every $e = (v, Z_i)$ in $T$, with $v \in X
  \cup Y$ and $i \in [3\ell]$, pick a copy $H_e$ of $H$ in $K[V] - A$ which
  exists as long as $|V \setminus A| \geq r(G,H)$. Label the vertices of $H_e$
  by $u_1,\dotsc,u_k$ and label $\{v\} \cup Z_i$ by $v_1,\dotsc,v_k$. Recall, by
  the choice of $V$, every pair of vertices in it has at least $m$ common
  neighbours, and in particular have at least $m/2 \geq k^4 \cdot r(G, H)/2$
  common neighbours \emph{outside of} $A$. For every pair $\{u_i, v_i\}$, we
  hence find a $u_iv_i$-switcher by, e.g.\ taking a copy of $H$ in their common
  neighbourhood and removing a vertex. After repeating this for every $i \in
  [3\ell]$, we add the newly found local absorber $L_e$ for $\{v\} \cup Z_i$
  (see Figure~\ref{fig:absorbers}) to $A$ and iterate the procedure from the
  beginning. To see that the order of $A$ is bounded as desired recall that $T$
  is a bipartite graph with $3\ell$ vertices in one vertex class and maximum
  degree $40$ and every local absorber is of order $k^2$, hence
  \[
    120\ell \cdot k^2 \leq 120|U|p \cdot k^2 = \frac{60 Np}{k^3} \cdot k^2 <
    \frac{\sqrt{N}}{2k} = \frac{m}{2}.
  \]

  It remains to verify that $A$ is indeed an absorber. Consider any $R \subseteq
  U$ of order $|R| \leq r(G, H)$. Recall, by the choice of $X$, every $v \in R$
  has at least $\sqrt{N}/(2^9 k^4) \geq r(G,H) + |R|(k-1)$ neighbours in $X$ so
  we can, for each $v \in R$ iteratively, find a copy of $H$ containing $v$ in
  $X$ that is disjoint from the previously found ones, until $R$ is covered. We
  proceed by finding copies of $H$ in the remainder of $X$ for as long as it has
  at least $\ell$ and at most $\ell + |H| - 1$ vertices left. Indeed, this can
  be done as $\ell \geq r(G,H)$.

  Finally, choose any $\ell$ vertices among the remaining ones in $X$, with a
  leftover of at most $|H|-1$. By looking at the corresponding matching in the
  template graph $T$, for every edge in the matching $e = (v,Z_i)$, we use the
  $H$-tiling of $L_e \cup \{v\} \cup Z_i$ in $K$ and for every edge $f$ not in
  the matching we use the $H$-tiling of $L_f$ in $K$ (see again
  Figure~\ref{fig:absorbers}). This completes the proof.
\end{proof}

\subsection{Efficient absorbers for bounded degree graphs}

Let $G$ and $H$ now be $k$-vertex graphs with maximum degree $\Delta$ and $k =
\max\{|G|,|H|\}$. Note that, the general absorbing lemma from the previous
section requires the host graph to be of the order $O(k^{10}r(G,H)^2)$. In fact,
any `naive' way of building an absorber, that is finding local absorbers and
stitching them in a clever way, would impose $N \geq k^2 \cdot r(G,H)$. If we
are to recover the bounds promised in Theorem~\ref{thm:main-symmetric} and
Theorem~\ref{thm:main-asymmteric-bounded-deg}, this is infeasible. In order to
circumvent this, we need to develop more efficient (smaller) absorbers. This is
the most novel part of our argument.

\begin{lemma}\label{lem:efficient-absorbers}
  Let $G$ and $H$ be two graphs with maximum degree $\Delta$ and $k = \max\{|G|,
  |H|\}$. Let $K$ be a graph of order $N \geq 2^{512\Delta\log^2\Delta}k$ whose
  complement is $G$-free. Then, there exists $X \subseteq V(K)$ of order at
  least $N/2^{200\Delta\log^2\Delta}$ which is
  bi-$(2^{-10\Delta-11}\Delta^{-2\Delta-4}, 2^{-4} \Delta^{-1})$-dense
  % $(32\Delta)^{-(2\Delta)}2^{-11}\Delta^{-4},(16\Delta)^{-1})$-dense
  and an $(H,r(G,H))$-absorber for $X$ in $K$ of order at most
  $2^{200\Delta\log^2\Delta} k$.
\end{lemma}
\begin{proof}
  We first introduce some parameters. Let
  \[
    \gamma = \frac{1}{32\Delta}, \qquad \eps =
    \frac{\gamma^{2\Delta}}{2^{11}\Delta^4}, \qquad s = \log(32\Delta), \qquad
    \beta = \frac{1}{8s},
  \]
  and note that $(1-\beta)^{2s} \geq 2/3$. Next, let
  \[
    m = 2^{90\Delta+12} \Delta^{33\Delta+6} k
  \]
  and so $\gamma^\Delta m \geq r(G,H) \cdot \Delta^6 2^8$ with room to spare
  (see \eqref{eq:ramsey-bound}). Lastly, take $m_0 = (\gamma/4) \cdot
  \eps^{s-1}\beta^{s-2} N$. The choice of parameters is set up in a way that, if
  we find a set of order $m_0 \geq 2^{-200\Delta\log^2\Delta} N$ which is
  bi-$(\eps,2\gamma)$-dense and an absorber for it of order at most $m \leq
  2^{200\Delta\log^2\Delta} k$, we are done.

  Since the complement of $K$ is $G$-free, it follows from
  Lemma~\ref{lem:greedy-embedding} that there are at least $N/2$ vertices with
  degree at least $\gamma N$ in $K$, and moreover, any $2^s\eps\beta \cdot 3m
  \geq 4k$ of these vertices induce a set of density at least $4\gamma$ (here we
  used $N/2 \geq 2^s\eps\beta \cdot 3m$). Therefore, by
  Lemma~\ref{lem:dense-to-bi-dense} applied to a subset of these high degree
  vertices of order $\eps^{-s+1}\beta^{-s+2} \cdot 3m$, there is a $V \subseteq
  V(K)$ of order $3m$ which is bi-$(\eps,2\gamma)$-dense and whose every vertex
  has degree at least $\gamma N$ in $K$.

  Consider a largest $2$-independent set of $H$, that is a set of vertices in
  which no two are within distance two in $H$, and denote it by $I$. It is an
  easy observation that $|I| \geq k/\Delta^3$. We first iteratively build a
  collection of disjoint copies of $H$ in $K[V]$ whose vertex set we move to a
  set $Q$. Initially $Q$ is empty, and for as long as $V \setminus Q$ is of
  order at least $2m + k$, we find a new copy of $H$ by applying
  Lemma~\ref{lem:embedding-lemma} to $K[V]$ with $I$ (also an independent set)
  and all $V_i = V \setminus Q$, and move its vertices to $Q$. In the end, $|Q|
  = m$ and the plan is for it to hold our absorber.

  Recall, in every found copy of $H$, the image of every vertex of $I$ comes
  with at least $2\gamma^\Delta m - k \geq \gamma^\Delta m$ aliases. However,
  plenty of these may actually belong to the set $Q$. This does not allow us to
  blindly use them for a different copy of $H$. To remedy this, for every copy
  of $H$ in $Q$ flip a fair coin independently to decide whether to put it into
  a set $Q_1$ or not. Consider some $v$ that had at least $3\gamma^\Delta m/4$
  aliases in $Q$. As choices for different copies of $H$ are independent, by
  Chernoff's inequality, $v$ has at least $\gamma^\Delta m/4$ aliases outside of
  $Q_1$ with probability at least $1 - e^{-\Omega(\gamma^\Delta m)}$. The union
  bound (over at most $m$ many vertices and then two events) shows that there
  exists a choice of $Q_1$ such that $2m/3 \geq |Q_1| \geq m/3$ and all vertices
  in all copies of $I$ in $Q_1$ have at least $\gamma^\Delta m/4$ aliases
  outside of $Q_1$. Denote the set of these by $W_I$, and the set of vertices
  belonging to the corresponding copies of $H - I$ by $W_{H-I}$. By removing
  some copies of $H$ if necessary, assume $|W_I \cup W_{H-I}| = m/4$.

  Let $W_A \subseteq V$ be a set of vertices outside of $W_I \cup W_{H-I}$ ($=
  Q_1$), each being an alias of at least $\gamma^\Delta |W_I|/2^6 \geq
  \gamma^\Delta m/(2^8\Delta^3)$ (using $|I| \geq k/\Delta^3$) vertices in
  $W_I$. In particular, each vertex in $W_A$ is an alias of at least a quarter
  of the average, and hence at most $|W_I| \cdot \gamma^{\Delta}m/2^4$ (a
  quarter) of all `$u \in V \setminus Q_1$ is an alias of $v \in W_I$' relations
  have $u \notin W_A$. Consequently, the set $W_I^\star \subseteq W_I$ of all
  vertices that each have at least $\gamma^\Delta m/2^4$ aliases inside $W_A$ is
  of order at least $|W_I|/2$. (Note that implicitly $|W_A| \geq \gamma^\Delta
  m/2^4$.)

  Let $W = W_I \cup W_{H-I} \cup W_A$. The following claim is crucial.

  \begin{claim}\label{cl:efficient-mini-absorber}
    For any $X \subseteq W_I^\star$ of order at most $\gamma^\Delta m/2^5$,
    there is an $H$-tiling of $W \setminus X$.
  \end{claim}
  \begin{proof}
    For every vertex $x \in X$ choose a distinct alias $w_x \in W_A$, out of
    $\gamma^\Delta m/2^4 - |X| \geq \gamma^\Delta m/2^5$ many, and place it into
    the corresponding copy of $H$ instead of $x$ (see
    Figure~\ref{fig:abs-proc-p2}). By doing this, the number of vertices every
    other $w \in W_A$ is an alias of \emph{remains unchanged}: if $w \in W_A$
    was an alias of $x$ then it remains an alias of $w_x$ as well. By a slight
    abuse of notation, call the remaining sets $W_I, W_H, W_A$ and note that
    $W_I$ remains of the same order as before, $W_I \cup W_H$ has an $H$-tiling,
    and $W_A$ shrunk by $|X|$.

    We now iteratively build an $H$-tiling by every time moving some $k$
    vertices $w_1,\dotsc,w_k$ from $W_A$ to $W_I$ and subsequently finding a
    copy of $H$ in $W_I$ and removing it, while maintaining the following
    invariant:
    \stepcounter{propcnt}
    \begin{alphenum}
      \item\label{p:h-tiling} the order of $W_I$ remains unchanged and $W_I
        \cup W_H$ has an $H$-tiling;
      \item\label{p:num-alias} the remaining vertices in $W_A$ are each still an
        alias of at least $\gamma^\Delta m/(2^8\Delta^3)$ vertices in $W_I$.
    \end{alphenum}
    Since $W_A$ shrinks by $k$ in every iteration, eventually we have fewer than
    $k$ vertices left in $W_A$ and by \ref{p:h-tiling} $W_I \cup W_H$ has an
    $H$-tiling, which completes the proof (see Figure~\ref{fig:abs-proc-p4}).

    Throughout the process, we apply Lemma~\ref{lem:embedding-lemma} to $K[V]$
    for carefully chosen sets as $V_i$: for every $w_i$ let $V_i \subseteq W_I$
    be the set of vertices $w_i$ is an alias of (see
    Figure~\ref{fig:abs-proc-p3}). By \ref{p:num-alias} each $V_i$ is of order
    at least $\gamma^\Delta m/(2^8\Delta^3)$ and thus $\gamma^\Delta \cdot
    \gamma^\Delta m/(2^8\Delta^3) - 2\Delta\eps m \geq k$.  Therefore, there
    exists a copy of $H$ in $K[W_I]$ that maps each vertex of $H$ into $v_i \in
    V_i$. As $w_i$ is an alias for $v_i$, by replacing each $v_i$ with $w_i$ we
    establish both \ref{p:h-tiling} and \ref{p:num-alias}.
  \end{proof}

  This property essentially enables us to use the set $W$ as an absorber, we
  just need to find a reservoir $U$ as in the statement of the lemma `to be
  absorbed'. Recall, $W_I^\star \subseteq V$ consists of vertices with degree at
  least $\gamma N$ in $K$, and thus at least $\gamma N/2$ into $V(K) \setminus
  W$. By a simple density argument, there is $\tilde V \subseteq V(K) \setminus
  W$ of order at least $\gamma N/4$ so that every $v \in \tilde V$ has at least
  $\gamma|W_I^\star|$ neighbours in $W_I^\star$. Now, similarly as in the
  beginning of the proof (using that the complement of $K$ is $G$-free), by
  Lemma~\ref{lem:dense-to-bi-dense} there is $U \subseteq \tilde V$ of order at
  least $m_0 = (\gamma/4) \cdot \eps^{s-1}\beta^{s-2} N$ which is additionally
  bi-$(\eps,2\gamma)$-dense.

  \begin{figure}[!htbp]
    \captionsetup[subfigure]{textfont=scriptsize}
    \centering
    \begin{subfigure}{.49\textwidth}
      \centering
      \includegraphics[scale=0.8]{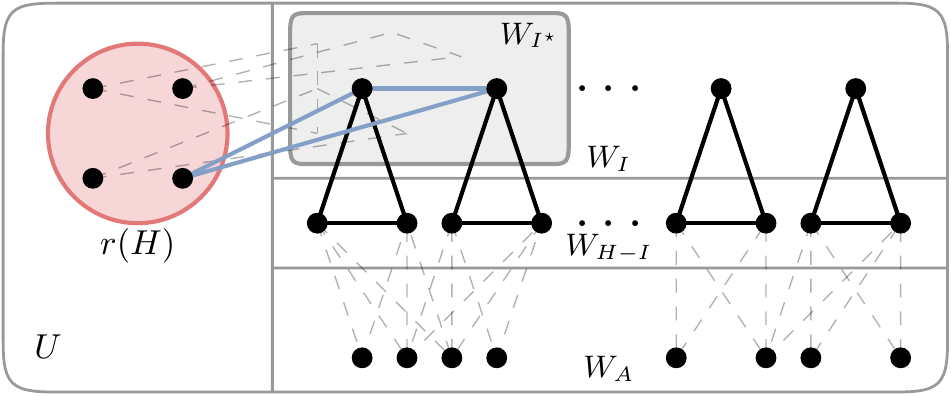}
      \caption{Leftover vertices in $U$ get absorbed separately by $W_I^\star$}
      \label{fig:abs-proc-p1}
    \end{subfigure}%
    \hfill
    \begin{subfigure}{.49\textwidth}
      \centering
      \includegraphics[scale=0.8]{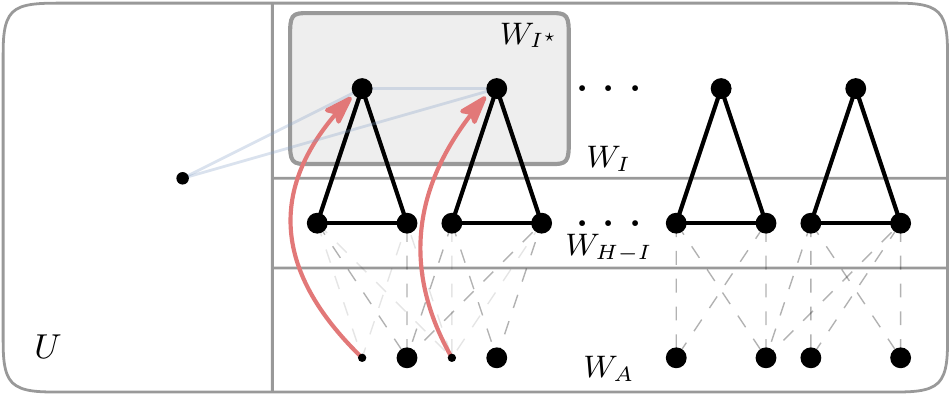}
      \caption{Aliases of used vertices replenish the set $W_I^\star$}
      \label{fig:abs-proc-p2}
    \end{subfigure}%
    \vspace{1em}
    \begin{subfigure}{.49\textwidth}
      \centering
      \includegraphics[scale=0.8]{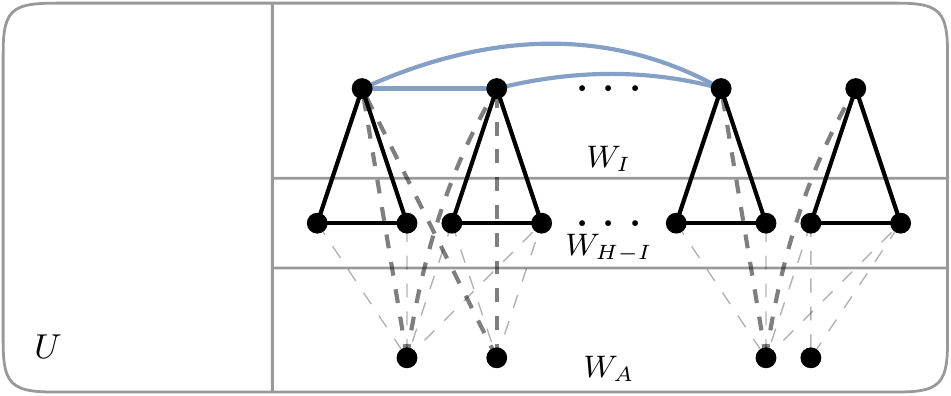}
      \caption{A copy of $H$ is found in $W_I$ for a set of $k$ aliases}
      \label{fig:abs-proc-p3}
    \end{subfigure}%
    \hfill
    \begin{subfigure}{.49\textwidth}
      \centering
      \includegraphics[scale=0.8]{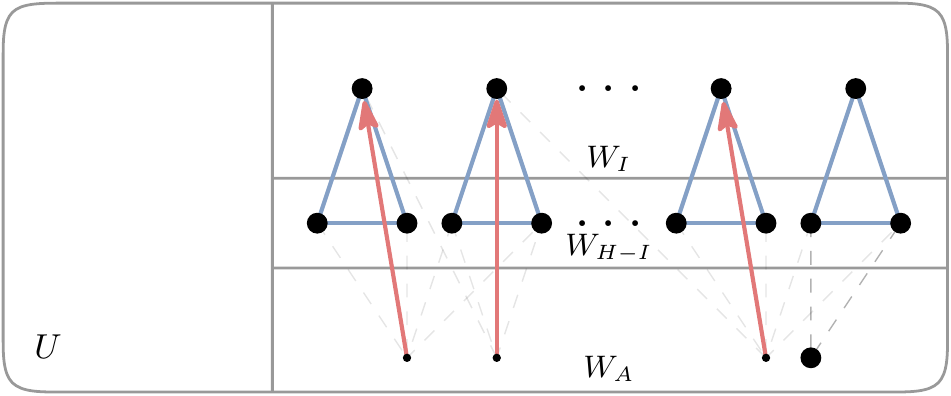}
      \caption{An $H$-tiling in $W_I \cup W_{H-I}$ completes the procedure}
      \label{fig:abs-proc-p4}
    \end{subfigure}%
    \caption{An illustration of the absorbing procedure for $H = K_3$}
    \label{fig:absorbing-procedure}
  \end{figure}

  We claim that $K[W]$ is an absorber for $U$. To establish this, consider any
  set $X \subseteq U$ of order at most $r(G,H)$. We find an $H$-tiling of $K[W
  \cup X]$ iteratively, by selecting some set $v_1, \dotsc, v_{|I|}$ still
  remaining in $X$, mapping $I$ to it, and extending it to a copy of $H$ by
  finding $H-I$ in $K[W_I^\star]$ disjoint from previously found ones (see
  Figure~\ref{fig:abs-proc-p1}). This is done through
  Lemma~\ref{lem:embedding-lemma} in $K[V]$ by setting $V_i$ to be $N_K(v_j,
  W_I^\star)$, for every $i \in N_H(j)$ with $j \in I$, and otherwise setting
  $V_i = W_I^\star$, both after removing previously found copies of $H$.
  Crucially, as $I$ is a $2$-independent set, just looking at disjoint
  neighbourhoods of vertices in $I$ is sufficient (namely, no common
  neighbourhood is necessary). As $\gamma^{\Delta} \cdot (\gamma|W_I^\star| -
  r(G,H) \cdot \Delta^3) - 2\Delta\eps m \geq k$ the requirements for the
  embedding are fulfilled. We repeat this until all vertices of $X$ are covered
  by some copies of $H$; in the very last iteration we may use fewer than $|I|$
  vertices, but then these are just used as the images of some $I' \subseteq I$
  instead of whole $I$. Lastly, let $Z \subseteq W_I^\star$ be the set used to
  find these copies of $H$. As $|Z| \leq r(G,H) \cdot \Delta^3 \leq
  \gamma^\Delta m/2^5$ with room to spare,
  Claim~\ref{cl:efficient-mini-absorber} completes the proof.
\end{proof}

At this point we believe the `additional' multiplicative $\log\Delta$ factor in
the exponent of $N$ (compared to the state-of-the-art bound for bounded degree
graphs, e.g.~\eqref{eq:ramsey-bound}) warrants some explanation. Namely, as in
the proof of Graham, R\"{o}dl, and Ruci\'{n}ski, the best density we can hope to
get in subgraphs of $K$ via Lemma~\ref{lem:greedy-embedding} is of order
$1/\Delta$. Then, for Lemma~\ref{lem:embedding-lemma} to make sense, the
parameter $\eps$ in it has to be of order $\Delta^{-\Delta}$. Lastly,
Lemma~\ref{lem:dense-to-bi-dense} requires $N \geq \eps^s k$ to give us a
(large) bi-$(\eps,\Omega(1/\Delta))$-dense set we can apply
Lemma~\ref{lem:embedding-lemma} to. As $s$ has to be roughly $\log\Delta$, this
forces $N = 2^{O(\Delta\log^2\Delta)} k$.

%% file: asymmetric_case.tex
\section{Asymmetric case: the long term behaviour of $r(G,nH)$}
\label{sec:asymmetric-problem}

We start with the significantly easier asymmetric problem and provide proofs of
the upper bounds for Theorem~\ref{thm:main-asymmetric} and
Theorem~\ref{thm:main-asymmteric-bounded-deg}, making use of our absorbing
lemmas: Lemma~\ref{lem:absorbing-lemma-general} and
Lemma~\ref{lem:efficient-absorbers}, respectively. The proofs are canonical and
almost identical to the argument of Buci\'{c} and Sudakov, which itself traces
back to a paper of Burr~\cite{burr1987ramsey}. We provide them for completeness.

\begin{theorem}\label{thm:asymmetric-general-ub}
  Let $G$ and $H$ be two graphs and $k = \max\{|G|,|H|\}$. Then, provided $n
  \geq 2^{32}k^{10}r(G,H)^2$, we have
  \[
    r(G, nH) \leq n|H| + r(\cD(G), H) - 1.
  \]
\end{theorem}
\begin{proof}
  Consider a $2$-coloured complete graph $K$ of order $N = n|H| + r(\cD(G),H) -
  1$ and assume, towards contradiction, that it does not contain a red $G$ nor a
  blue $nH$. Therefore, by Lemma~\ref{lem:absorbing-lemma-general} applied to
  the blue subgraph of $K$, there is a set $B \subseteq V(G)$ of order
  $N/(2k^3)$ and an $(H,r(G,H))$-absorber $A$ for $B$.

  Next, choose a maximal collection of disjoint blue copies of $H$ in $K - (A
  \cup B)$, denote the union of their vertex sets by $C$, and set $D := K - (A
  \cup B \cup C)$. Observe that $|D| \leq r(G, H)$ since the graph induced by it
  does not contain a red $G$ nor a blue $H$.

  Provided it exists, find a vertex $v \in D$ that has at least $r(G, H)$ blue
  neighbours in $B$. We know such a vertex has a blue copy of $H$ in its
  neighbourhood in $B$. Replace an arbitrary vertex in it by $v$ and add the
  vertices of this newly found blue copy of $H$ (the one including $v$) to $C$.
  Let $B'$, $C'$, and $D'$ denote the resulting sets after repeating this
  process for as many times as possible. At this point,
  \[
    |B'| \geq |B| - (k-1)|D| \geq |B| - kr(G,H) \geq \frac{N}{2k^3} - kr(G, H)
    \geq k^3 r(G, H),
  \]
  the vertices in $C'$ can still be perfectly tiled, and all vertices in $D'$
  have fewer than $r(G, H)$ blue neighbours in $B'$.

  Furthermore, we claim there is no member of $\cD(G)$ in $D'$ in red. To see
  this, choose any graph in $\cD(G)$ and suppose there is such a red copy, call
  it $G'$. Since there are at most $|G'| r(G,H) \leq k r(G, H)$ blue neighbours
  of the vertices in $G'$ in $B'$ and $|B'| \geq kr(G,H) + \alpha(G)$, $G'$ has
  more than $\alpha(G)$ red neighbours in $B'$, which can be used to complete a
  red copy of $G$. We conclude $|D'| \leq r(\cD(G), H) - 1$. This, by the
  assumption on the order of $K$, implies
  \[
    |A \cup B' \cup C'| \geq N - r(\cD(G), H) + 1 \geq n|H|.
  \]
  It suffices to prove that we can tile $A \cup B' \cup C'$ with copies of $H$
  in blue. As $C'$ is already perfectly tiled it remains to tile $A \cup B'$.

  For this we find a maximal disjoint collection of blue copies of $H$ in $B'$
  and denote the remaining set by $R$. Note that $|R| \leq r(G, H)$, and so by
  the absorbing property of $A$, we find an $H$-tiling of $K[A \cup R]$ in blue,
  as desired. This yields a blue $nH$ and is a contradiction with the initial
  assumption.
\end{proof}

The argument for when $G$ and $H$ are of bounded maximum degree is very similar,
except for how we cover $D$. This is done with the help of
Lemma~\ref{lem:embedding-lemma}.

\begin{theorem}
  Let $G$ and $H$ be two graphs with maximum degree $\Delta$ and $k =
  \max\{|G|,|H|\}$. Then, provided $n \geq 2^{O(\Delta\log^2\Delta)}k/|H|$, we
  have
  \[
    r(G,nH) \leq n|H| + r(\cD(G),H) - 1.
  \]
\end{theorem}
\begin{proof}
  Consider a $2$-coloured complete graph $K$ of order $N = n|H| + r(\cD(G), H) -
  1$ and assume, towards contradiction, that it does not contain a red $G$ nor a
  blue $nH$. Therefore, by Lemma~\ref{lem:efficient-absorbers} applied to the
  blue subgraph of $K$, there is a set $B \subseteq V(K)$ of order at least
  $N/2^{200\Delta\log^2\Delta} \geq 2^{O(\Delta\log^2\Delta)}k$ and an
  $(H,r(G,H))$-absorber $A$ for $B$. Moreover, $B$ is bi-$(\eps,\gamma)$-dense,
  for $\gamma = 1/(16\Delta)$ and $\eps = (\gamma/2)^{2\Delta} \cdot
  1/(2^{11}\Delta^4)$.

  Next, choose a maximal collection of disjoint blue copies of $H$ in $K - (A
  \cup B)$, denote the union of their vertex sets by $C$, and set $D := K - (A
  \cup B \cup C)$. Observe that $|D| \leq r(G, H) \leq 2^{128\Delta\log\Delta}k$
  (see \eqref{eq:ramsey-bound}), since the graph induced by it does not contain
  a red $G$ nor a blue $H$.

  Let $I$ be a largest $2$-independent set of $H$; recall, $|I| \geq
  k/\Delta^3$. Provided they exist, find $|I|$ vertices $v_1,\dotsc,v_{|I|}$ in
  $D$ each of which has a blue degree of at least $|B|/(2\Delta$) in $B$. For as
  long as there are $|I|$ such vertices (or, in the single instance there are
  fewer just work with these), we extend them to a copy of $H$ by finding $H -
  I$ in $B$, and afterwards move it to $C$. More precisely, we apply
  Lemma~\ref{lem:embedding-lemma} with the blue neighbourhood of $v_j$ in $B$ as
  $V_i$ for every $i \in N_H(j)$ with $j \in I$ and otherwise setting $V_i = B$
  (both after removing previous copies of $H$). This can be done as, crucially,
  no two neighbourhoods of vertices in $I$ intersect, $B$ is
  bi-$(\eps,\gamma)$-dense, and
  \begin{align*}
    \Big(\frac{\gamma}{2}\Big)^{\Delta-r}\Big(\frac{|B|}{2\Delta} -
    k\frac{|D|}{|I|}\Big) - 2r\eps|B|
    & \geq \Big(\frac{\gamma}{2}\Big)^{\Delta}\Big(\frac{|B|}{2\Delta} -
    \Delta^3|D|\Big) - 2\Delta \eps|B| \\
    & \geq
    \Big(\frac{\gamma}{2}\Big)^{\Delta}\frac{|B|}{2\Delta} - o(|B|)\footnotemark
    - 2\Delta\eps|B| \geq k.
  \end{align*}

  \footnotetext{Here is one particular example of how we abuse the $o(\cdot)$
    notation in an attempt to increase readability. We implicitly mean that
    there is a choice of the constant in $N \geq 2^{O(\Delta\log^2\Delta)}k$ so
  that $\Delta^3 r(G,H)$ is much smaller than $|B|$.}

  Let $B'$, $C'$, and $D'$ denote the resulting sets after this process. At this
  point, $|B'| \geq |B| - k|D|/|I| = (1-o(1))|B|$, the vertices in $C'$ can
  still be perfectly tiled, and all vertices in $D'$ have fewer than
  $|B|/(2\Delta)$ blue neighbours in $B'$.

  Furthermore, we claim there is no member of $\cD(G)$ in $D'$ in red. To see
  this, choose any graph in $\cD(G)$ and suppose there is such a red copy, call
  it $G'$. Note that every vertex in the independent set $G - G'$ has at most
  $\Delta$ neighbours in $G'$. Since no vertex of $D'$ has more than
  $|B|/(2\Delta)$ blue neighbours in $B'$, they all have at least $|B'| -
  |B|/(2\Delta)$ red neighbours, and in particular, any set of at most $\Delta$
  vertices in $D'$ has more than $k$ red common neighbours in $B'$. Therefore,
  we can greedily extend $G'$ to a red copy of $G$. We conclude $|D'| \leq
  r(\cD(G),H) - 1$. This, by the assumption on the order of $K$, implies
  \[
    |A \cup B' \cup C'| \geq N - r(\cD(G), H) + 1 \geq n|H|.
  \]
  A perfect $H$-tiling of $A \cup B' \cup C'$ is obtained via the absorbing
  property of $A$, analogously as before. This completes the proof.
\end{proof}

%% file: symmetric_case.tex
\section{Symmetric case: Ramsey numbers of $nH$ for sparse graphs}
\label{sec:symmetric-problem}

In this section we give the proofs of our results concerning Ramsey numbers of
$nH$. The focal point is establishing almost optimal, that is up to a
logarithmic factor in the exponent, bounds on $n$ for when $H$ is a graph of
maximum degree $\Delta$. Throughout we think of $H$ being a $k$-vertex graph and
let $\alpha := \alpha(H)$.

\subsection{Graphs with bounded degree}

We begin by listing a loose lower bound due to Burr, Erd\H{o}s, and
Spencer~\cite{burr1975ramsey} which holds for any $n$ and (almost) any $H$.

\begin{proposition}\label{prop:symmetric-lower-bound}
  Let $H$ be a $k$-vertex graph with no isolated vertices. Then, for any $n$
  \[
    r(nH) \geq n(2k-\alpha) - 1.
  \]
\end{proposition}
\begin{proof}[Sketch of the proof.]
  Partition a complete graph on $N = (2k-\alpha)n - 2$ vertices into $R$
  and $B$, of order $|R| = (k-\alpha)n-1$ and $|B| = kn-1$. Colour all
  the edges in $R$ with red, in $B$ with blue, and in $[R,B]$ with red.
\end{proof}

A crucial concept for studying Ramsey numbers of multiple copies of a graph is
that of a \emph{tie} (see \cite{burr1987ramsey} as well as \cite{burr1975ramsey}
under the name `bowtie'). A $2$-coloured graph on $2k-\alpha$ vertices is an
\emph{$H$-tie} if it contains both a blue $H$ and a red $H$, and is minimal with
respect to the number of edges. In a certain sense it serves as a mini-absorber
for $nH$: namely, if we can find $n-1$ monochromatic copies of $H$ in the same
colour, then we can use the $H$-tie to complete all $n$ copies. Hence, if $n$ is
large enough for us to find an $H$-tie, assuming there is no monochromatic $nH$,
we immediately have $r(nH) \leq (2k-\alpha)n + r((n-1)H)$. However, determining
the point $n_0$ when the long term behaviour in Theorem~\ref{thm:main-symmetric}
kicks in is far trickier than just determining how large $n$ has to be for an
$H$-tie to exist (see discussion in \cite[Section~4.2]{bucic2021tight}).
Nevertheless, the latter assertion (in a more complicated setting, as in the
upcoming lemma) is one of at least two places in our proof that require $n$ to
be of the order $2^{O(\Delta\log^2\Delta)}k$.

\begin{lemma}\label{lem:H-tie-lemma}
  Let $H$ be a $k$-vertex graph with maximum degree $\Delta$. Given a
  $2$-coloured complete graph $K$ on vertices $R \cup B$ such that $K[R]$ has no
  blue $H$, $K[B]$ has no red $H$, and $|R|, |B| \geq
  2^{256\Delta\log^2\Delta}k$, there exists an $H$-tie with red part in $R$ and
  blue part in $B$.
\end{lemma}
\begin{proof}
  We first introduce some parameters. Let
  \[
    \gamma = \frac{1}{32\Delta}, \qquad s = \log(32\Delta), \qquad \beta =
    \frac{1}{8s}, \qquad \eps' =
    \Big(\frac{\gamma}{4}\Big)^{\Delta}\frac{1}{8\Delta}, \qquad \text{and}
    \qquad \eps = \gamma^\Delta\frac{\eps'}{8\Delta}.
  \]
  Let $N = 2^{256\Delta\log^2\Delta}k$. As per usual, the assumption on
  complements being $H$-free implies (via Lemma~\ref{lem:greedy-embedding}) $R$
  is $((2\eps)^s\beta^{s-1},4\gamma)$-dense in red, and similarly $B$ in blue.
  This sets us up for an application of Lemma~\ref{lem:dense-to-bi-dense} to
  obtain $R' \subseteq R$ and $B' \subseteq B$, both of order $m \geq
  \eps^{s-1}\beta^{s-2}N$, and both bi-$(\eps,2\gamma)$-dense in their
  respective colours. Note $m \geq 2\eps^{-1}r(H)$
  (see~\eqref{eq:ramsey-bound}).

  If $K[R', B']$ is not bi-$(\eps',2\gamma)$-dense in one of the colours, say
  red, then we find an $H$-tie as follows. Let $R_1$ and $B_1$ be sets of order
  $\eps' m$ with red density less than $2\gamma$. Let $B_2 \subseteq B_1$ be the
  vertices whose red degree into $R_1$ is at most $4\gamma|R_1|$, which implies
  $|B_2| \geq |B_1|/2 \geq r(H)$. Therefore, there is a blue copy $H'$ of $H-I$
  in $B_2$. As every $i \in I$ has at most $\Delta$ neighbours in $H$ and every
  vertex in $B_2$ has at most $4\gamma|R_1|$ red neighbours in $R_1$, every set
  of $\Delta$ vertices in $H'$ has at least $(1-4\Delta\gamma)|R_1| \geq
  |R_1|/2$ blue common neighbours in $R_1$ that can be used to map $i$ to. Use
  these as $V_i$ for every $i \in I$, and set $V_j = R_1$ for all $j \notin I$.
  We apply Lemma~\ref{lem:embedding-lemma} to $K[R']$ and since
  \[
    \gamma^\Delta \cdot \eps' m/2 - 2\Delta\eps m \geq k
  \]
  we obtain a red copy of $H$ in $R'$ that together with $H'$ forms an $H$-tie.

  Otherwise, $K[R',B']$ is bi-$(\eps',2\gamma)$-dense in both colours and so a
  simple (but tedious) computation implies that the graph obtained as the union
  of red edges in $R$', red edges between $R'$ and $B'$, and blue edges in $B'$,
  is bi-$(2\eps',\gamma/2)$-dense. To obtain an $H$-tie we now find a red $H -
  I$ in $R'$ and, joined to it by red edges, a blue $H$ in $B'$. This is done by
  applying Lemma~\ref{lem:embedding-lemma} with $\gamma/4$ (as $\gamma$) to this
  graph with $V_i = R'$ for all $i \in H-I$ in the red copy of $H$ and $V_j =
  B'$ for $j \in H$ in the blue copy of $H$. The fact that
  \[
    (\gamma/4)^\Delta m - 4\Delta\eps' m \geq 2k
  \]
  confirms that this can be done.
\end{proof}

With this, we can show that there exists a \emph{critical colouring} for $nH$,
that is a colouring of a complete graph on $r(nH)-1$ vertices avoiding a
monochromatic $nH$, that contains a lot of structure. In fact, it is `canonical'
in a way that it is almost identical to the colouring from the lower bound in
Proposition~\ref{prop:symmetric-lower-bound}, with an addition of a rather small
exceptional set $E$.

\begin{theorem}\label{thm:structure-in-critical-colouring}
  Let $H$ be a $k$-vertex graph with no isolated vertices and maximum degree
  $\Delta$. Then, provided $n \geq 2^{O(\Delta\log^2\Delta)} k$, there exists a
  $2$-colouring of $K_{r(nH)-1}$ with no monochromatic $nH$ and the following
  structure. There is a partition of the vertex set into $R$, $B$, and $E$, such
  that:
  \begin{itemize}
    \item $|E| \leq r(H)$ and $|R|, |B| \geq k(|E|+1)$;
    \item all edges inside $R$ are red and all edges inside $B$ are blue;
    \item all edges between $R$ and $B$ are of the same colour, all edges
      between $R$ and $E$ are blue, and all edges between $B$ and $E$ are red;
    \item there is no $H$-tie containing a vertex of $E$.
  \end{itemize}
\end{theorem}
\begin{proof}
  Consider an arbitrary critical $2$-colouring $c$ of the complete graph $K$ on
  $N = r(nH)-1$ vertices. Let $\alpha = \alpha(H)$, $I$ be an independent set
  achieving it, and $\hat H = H - I$. Recall that $\alpha$ satisfies
  $\frac{k}{\Delta+1} \leq \alpha \leq (1-\frac{1}{\Delta+1})k$. Let $q :=
  2^{256\Delta\log^2\Delta} k$, having in mind that we could apply
  Lemma~\ref{lem:H-tie-lemma} to sets of order $q$.

  Proposition~\ref{prop:symmetric-lower-bound} implies $N \geq n(2k-\alpha) - 2
  \geq nk + nk/(\Delta+1) - 2$. Remove a maximal collection of disjoint red
  copies of $H$ and note that the remainder is of order at least $nk/(2\Delta)$.
  Take a half of these vertices and, as they form no blue $H$, apply
  Lemma~\ref{lem:efficient-absorbers} to the red subgraph induced by it to
  obtain a set $R_0$ with $|R_0| \geq nk/(4\Delta) \cdot
  2^{-200\Delta\log^2\Delta} \geq nk/2^{203\Delta\log^2\Delta}$ and an
  $(H,r(H))$-absorber $A_r$ for it of order at most $q/2^8$. Similarly, by
  taking a largest collection of red copies of $H$, we find a set $B_0$ disjoint
  from $R_0$ of order at least $nk/2^{203\Delta\log^2\Delta}$ and an
  $(H,r(H))$-absorber $A_b$ for it of order at most $q/2^8$. If
  necessary, remove some vertices for $|R_0| = |B_0|$ to hold (for technical
  reasons).

  Since $R_0$ does not contain a blue $H$, $B_0$ does not contain a red $H$, and
  $|R_0|, |B_0| \geq q$, we can apply Lemma~\ref{lem:H-tie-lemma} to find an
  $H$-tie with red part in $R_0$ and blue part in $B_0$. Repeat doing this for
  as long as possible, and afterwards keep only the $H$-ties in the majority
  colour, say red. This majority determines the colour between
  final sets $[R,B]$ as promised by the theorem, but more importantly the colour
  of the absorber we need for the rest of the proof. So, at this point the
  `whole' copies of $H$ lie in the blue part and the $H$-absorber in red is
  useless (if the majority is blue we discard the blue absorber).

  Let $R_1$ be the union of $k-\alpha$ vertices of each of these ties that
  belong to $R_0$, and likewise $B_1$ the union of $k$ vertices of these ties
  that belong to $B_0$. Consequently, $|B_1| \geq (|B_0|-q)/2$ and $|R_1| =
  |B_1|/k \cdot (k-\alpha)$, and so, for some $m_1 \geq
  n/2^{205\Delta\log^2\Delta}$, we have $|R_1| = (k-\alpha)m_1$ and $|B_1| =
  km_1$.

  From what remains in $K - (R_1 \cup B_1 \cup A_b)$ extract a maximal
  collection of disjoint $H$-ties into a set $T$, and denote the remainder by
  $E_1$. A property used several times in what is to come is that a
  $\frac{k}{2k-\alpha}$-fraction of $T$ can always be covered by copies of $H$
  in the same colour (either all red or all blue).

  \begin{claim}\label{cl:exceptional-set-size}
    $|E_1| \leq 12\Delta q = o(m_1)$\footnote{This is another instance of
      abusing the asymptotic notation. Since $12\Delta q$ does not depend on
      $n$, and $m_1 \geq n/2^{122\Delta\log^2\Delta}$, we implicitly mean here
      that we can always choose the constant in the term for $n$ so that
    $12\Delta q$ is much smaller than $m_1$.}.
  \end{claim}
  \begin{proof}
    Suppose towards contradiction this is not the case. We first show that if we
    can cover at least $\frac{k}{2k-\alpha}(|E_1| + |A_b|)$ vertices of $E_1$
    with copies of $H$ in one of the colours then there exists an $nH$ in that
    colour. Indeed, by construction we can always cover a
    $\frac{k}{2k-\alpha}$-fraction of both $R_1 \cup B_1$ and $T$ by copies of
    $H$ in the same colour. Putting it together, this means we can find at least
    $N/(2k-\alpha) \geq n - 2/(2k-\alpha) > n-1$ copies of $H$ all in one
    colour---a contradiction.

    In case such a cover for $E_1$ does not exist, while also using $\alpha \leq
    (1-\frac{1}{\Delta+1})k$, there is an $R' \subseteq E_1$ of order
    \[
      |E_1| - \frac{k}{2k-\alpha}(|E_1| + |A_b|) \geq
      \frac{k-\alpha}{2k-\alpha}|E_1| - |A_b| \geq \frac{|E_1|}{4\Delta} -
      \frac{q}{2^8} \geq 2q
    \]
    that does not contain a blue $H$. Analogously, there is a $B' \subseteq E_1$
    of order at least $2q$ without a red $H$. Choosing disjoint subsets of $R'$
    and $B'$ if necessary (and in the process losing at most a half of vertices
    in each), we can apply Lemma~\ref{lem:H-tie-lemma} to find an $H$-tie
    contradicting the definition of $E_1$.
  \end{proof}

  Set $Q := \varnothing$ and for as long as there is a vertex in $R_1$ with blue
  degree at least $r(H)$ into $B_1 \setminus Q$, find a blue copy of $H$ that
  contains this vertex and otherwise lies in $B_1 \setminus Q$, and move its
  vertices to $Q$. We claim that $Q$ cannot become too large in this process.

  \begin{claim}\label{cl:red-degree-across}
    $|Q| \leq k|E_1|$.
  \end{claim}
  \begin{proof}
    Suppose towards contradiction this is not the case. In particular, there are
    $|E_1|$ copies of blue $H$ in $Q$, each of which has exactly one vertex in
    $R_1$ and exactly $k-1$ vertices in $B_1$. By the absorbing property of
    $A_b$ and as $B_1$ has no red $H$, we can cover the vertices remaining in
    $B_1 \cup A_b$ with blue copies of $H$ up to a leftover of order $k-1$. In
    total, we have covered at least $|E_1| + |B_1| + |A_b| - (k-1)$ vertices of
    $R_1 \cup B_1 \cup A_b$ by blue copies of $H$. Additionally, a
    $\frac{k}{2k-\alpha}$-fraction of vertices in $T$ can be covered by blue
    copies of $H$ by construction. Using the fact that $|B_1|/k =
    |R_1|/(k-\alpha)$, we get at least
    \begin{multline*}
      |E_1| + |B_1| + |A_b| - (k-1) + \frac{k}{2k-\alpha}|T| \geq \\
      |E_1| + |A_b| - (k-1) + \frac{k}{2k-\alpha}\big(|T| + |B_1 \cup
      R_1|\big) \geq \frac{k}{2k-\alpha} N > k(n-1)
    \end{multline*}
    (here we implicitly use that $\frac{k-\alpha}{2k-\alpha}|A_b| \geq k-1$)
    vertices covered with blue copies of $H$---a contradiction.
  \end{proof}

  Let $R_2$ and $B_2$ denote what remains in $R_1$ and $B_1$ after removing $Q$,
  and $E_2 := E_1 \cup Q$. Note that $|R_2| \geq |R_1| - |E_1| \geq
  (1-o(1))|R_1|$ and similarly $|B_2| \geq (1-o(1))|B_1|$. Therefore, $R_2$ is
  sufficiently large (and still blue $H$-free) for
  Lemma~\ref{lem:efficient-absorbers} to provide a set $R_3 \subseteq R_2$ of
  order at least
  \[
    |R_2|/2^{200\Delta\log^2\Delta} \geq |R_1|/2^{201\Delta\log^2\Delta} =
    (k-\alpha) m_1/2^{201\Delta\log^2\Delta}
  \]
  and an $(\hat H,r(H,\hat H))$-absorber for it of order at most $q/2^8$. With a
  slight abuse of notation, we let $A_r$ stand for this absorber from now on
  (recall, the $H$-absorber we previously had in red is gone). Set $B_3 := B_2$
  and move the remaining vertices of $R_2$ (the ones not used for $R_3$ and
  $A_r$) into $E_2$, By additionally removing arbitrary vertices from both $R_3$
  and $B_3$ if necessary, suppose $|R_3| = (k-\alpha)m$ and $|B_3| = km$, for
  some $m \geq m_1/2^{201\Delta\log^2\Delta}$. In particular, $r(H) = o(m)$
  (again abusing asymptotic notation and recalling~\eqref{eq:ramsey-bound}). Let
  $E_3$ be what remains of $E_2$ after moving all $H$-ties within it into $T$.

  To recapitulate, we have obtained a partition of $K$ into $R_3 \cup A_r \cup
  B_3 \cup A_b \cup E_3 \cup T$, where $R_3$ and $B_3$ are still reasonably
  large, no vertex in $R_3 \cup A_r$ has more than $r(H)$ blue neighbours into
  $B_3$, $T$ is a collection of $H$-ties, and $E_3$ contains no $H$-ties.

  \begin{claim}\label{cl:exceptional-set-size-again}
    $|E_3| \leq 12\Delta q = o(m)$.
  \end{claim}
  \begin{proof}
    The proof is identical to that of Claim~\ref{cl:exceptional-set-size}, with
    one key difference. Since $A_r$ is now an $(\hat H,r(H,\hat H))$-absorber
    for $R_3$, we need to show that we can cover a
    $\frac{k}{2k-\alpha}$-fraction of $R_3 \cup B_3$ by red copies of $H$ (for
    blue copies this is identical to what we did before). As a start, by the
    absorbing property, we can cover all but at most $k-1$ vertices in $R_3 \cup
    A_r$ by red copies of $\hat H$. We would like to extend every such copy to a
    red $H$ by using $\alpha$ vertices of $B_3$, which would imply we can cover
    at least
    \[
      |R_3| + \alpha \cdot \frac{|R_3|}{k-\alpha} \geq \frac{k}{2k-\alpha}|R_3
      \cup B_3|
    \]
    vertices with red copies of $H$ as desired. Recalling the definition of $Q$,
    there are at most $(k-\alpha)r(H)$ vertices in $B_3$ that have a blue edge
    towards any fixed set of $k-\alpha$ vertices in $R_3$. As in total we use at
    most
    \[
      \frac{|R_3|}{k-\alpha} \cdot \alpha \leq \frac{\alpha}{k} |B_3|,
    \]
    vertices, and $(1-\alpha/k)|B_3| = (k-\alpha)m \geq (k-\alpha)r(H)$ (with
    room to spare), we can greedily extend each such a copy of $\hat H$ to a red
    $H$.
  \end{proof}

  In the next phase, repeatedly add any $H$-tie from $R_3 \cup B_3 \cup E_3$
  that has at least one vertex in $E_3$ into $T$. Once we cannot do this,
  denote by $R_4$, $B_4$, and $E_4$ the remaining sets, respectively. Since
  $|E_3| = o(m)$, we have $|R_4| \geq |R_3| - o(km)$, $|B_4| \geq |B_3| -
  o(km)$, and $|E_4| \leq |E_3| \leq 12\Delta q$.

  The following claim tells us that the vertices in $E_4$ have `small' red
  degree into $R_4$ and blue degree into $B_4$, getting us almost to the desired
  structure.

  \begin{claim}\label{cl:neighbourhood-structure}
    Every $v \in E_4$ has red degree into $R_4$ and blue degree into $B_4$ both
    bounded by $q$.
  \end{claim}
  \begin{proof}
    Let $R'$ be the neighbours of $v$ in $R_4$ in red and similarly $B'$ the
    ones in $B_4$ in blue. Assume $|R'| \geq q$. Our goal is to show there is an
    $H$-tie containing $v$, contradicting the definition of $E_4$. If $|B'|
    \geq q$, then Lemma~\ref{lem:H-tie-lemma} gives us an $H$-tie with red part
    in $R'$ and blue part in $B'$. W.l.o.g.\ we may assume that the edges
    between are red. If so, we can take $v$ and $k-1$ vertices from the blue
    copy in $B'$ so that the independent set $I$ of the red copy of $H$ is
    contained inside of these $k-1$ vertices (possible as $\alpha < k$). Adding
    the $k-\alpha$ vertices from the red part in $R'$ forms an $H$-tie: it still
    contains a blue $H$ and a red $H-I$ in $R'$ that extends to a red $H$ into
    $B'$. In the remaining case, i.e.\ if $|B'| < q$, the red neighbourhood of
    $v$ in $B_4$ is larger than $q$ and Lemma~\ref{lem:H-tie-lemma} finds an
    $H$-tie in $R', B_4 \setminus B'$. By replacing a vertex in the red part by
    $v$ we again have an $H$-tie containing $v$---a contradiction.
  \end{proof}

  For ease of reference, for any vertex $v \in E_4$, we call its blue
  neighbourhood in $R_4$ and its red neighbourhood in $B_4$ its \emph{good}
  neighbourhood in $R_4, B_4$, respectively.

  With these preparations at hand, we are finally ready to define a new
  colouring $c'$ of $K$ and sets $R$, $B$, and $E$ as in the statement of the
  theorem. First, the exceptional set remains the same, that is $E := E_4$. Let
  $R$ be the union $R_4 \cup A_r$ and $k-\alpha$ vertices in each $H$-tie in
  $T$. Similarly, $B$ is defined as the union $B_4 \cup A_b$ and the remaining
  $k$ vertices in each $H$-tie in $T$. As for $c'$, we keep the same colouring
  of $E$ as in $c$, $R$ is coloured fully in red, $B$ fully in blue, and all the
  edges in $K[R, B]$ are red, $K[E, R]$ blue, and $K[E, B]$ red. Moreover, $|R|,
  |B| \geq (k-\alpha)m - o(km) \geq k(|E|+1)$. This is exactly what is promised
  by the theorem, it remains only to show that there is no $H$-tie using a
  vertex in $E$ and if there is a monochromatic $nH$ in $c'$ then there is one
  in original colouring as well. Lastly, these imply $|E| \leq r(H) = o(n)$, as
  otherwise a monochromatic copy of $H$ in $E$ together with either $k-\alpha$
  vertices in $R$ or $k-\alpha$ vertices in $B$ form an $H$-tie intersecting
  $E$.

  \begin{claim}\label{cl:no-H-tie-new-colouring}
    There is no $H$-tie in $c'$ that uses a vertex of $E$.
  \end{claim}
  \begin{proof}
    We show that if such a tie exists, then it existed in the original
    colouring, which would be a contradiction with the defining property of $E$.
    Let \emph{$H$-join} be a $2$-coloured graph consisting of a red copy of $H$,
    a disjoint blue copy of $H$, and a complete bipartite graph between them
    coloured red. Observe that an $H$-tie as above is then a subset of $E$ and
    an $H$-join with red part in $R$ and blue part in $B$. To see this, note
    that by definition of an $H$-tie any red edge in it belongs to the red copy
    of $H$, and therefore the red part of the $H$-tie in $R$ can be extended to
    a red copy of $H$. Analogous reasoning holds for blue. These copies of $H$
    are fully joined by red edges from $[R, B]$, completing the $H$-join as
    wanted.

    Let $E' \subseteq E$ be the vertices in this $H$-tie intersecting $E$, and
    so $|E'| \leq 2k-\alpha$. Let $R' \subseteq R_4$ and $B' \subseteq B_4$ be
    the common good neighbourhoods of $E'$ in the original colouring. Hence, by
    Claim~\ref{cl:neighbourhood-structure} it follows that $|R'| \geq |R_4| -
    2kq \geq (k-\alpha)m - o(km) - 2kq \geq (k+1)r(H)$ and $|B'| \geq |B_4| -
    2kq \geq (k+1)r(H)$. In view of Claim~\ref{cl:red-degree-across}, there is a
    red $H$ in $R'$ whose vertices all have at most $r(H)$ blue neighbours in
    $B'$. This implies their common red neighbourhood in $B'$ is at least $r(H)$
    which gives a blue copy of $H$ inside it completing the $H$-join as desired.
  \end{proof}

  Suppose there exists a monochromatic $nH$ in the new colouring $c'$. We claim
  that for any copy of $H$ in it, there is a corresponding monochromatic copy in
  $c$ using the same vertices in $E = E_4$ and the same number of vertices in
  $R_4$ and $B_4$ as it does in the new colouring in $R$ and $B$.

  Consider first a copy of $H$ intersecting $E$, and let $E'$ denote this
  intersection. Such a copy is contained in the union of $E'$ and an $H$-join
  with red part in $R$ and blue part in $B$ as before. Hence, it suffices to
  find for it an $H$-join in $c$ whose red and blue parts belong to the good
  common neighbourhoods of $E'$, and is additionally disjoint from other
  corresponding copies of $H$. Using Claim~\ref{cl:neighbourhood-structure},
  $E'$ has a common blue neighbourhood $R' \subseteq R_4$ of order at least
  $|R_4| - kq - k|E|$ in $R_4$, and a common red neighbourhood $B' \subseteq
  B_4$ of order $|B_4| - kq - k|E|$, both disjoint from other corresponding
  copies of $H$ (this is where $k|E|$ comes from). Since $|R_4| - kq - k|E| \geq
  r(H)$, there is a red $H$ in $R'$. Lastly, every vertex of such $H$ has at
  most $r(H)$ blue neighbours in $B_4$ by Claim~\ref{cl:red-degree-across}, and
  so at least $|B_4| - kq - k|E| - kr(H) \geq r(H)$ red common neighbours in
  $B'$. Thus, we find a blue $H$ in this common neighbourhood. As both
  colourings $c$ and $c'$ are the same on $E'$, there is a corresponding copy of
  $H$ in the original colouring as desired.

  Let us remove all copies of $H$ intersecting $E$ in the new colouring and all
  the corresponding copies of $H$ in the old colouring. In the old colouring,
  let $R_5$ and $B_5$ be the sets of remaining vertices from $R_4$ and $B_4$,
  and let $t$ be the number of $H$-ties in $T$. With this, all the remaining
  copies of $H$ from the monochromatic $nH$ in $c'$ now belong to $R \cup B$. In
  particular, taking into account the number of vertices used in the removed
  copies, these remaining copies of $H$ have only $r = |R_5| + |A_r| +
  t(k-\alpha)$ vertices of $R$ and $b = |B_5| + |A_b| + tk$ vertices of $B$ to
  be found among.

  Moreover, there are at most $\floor{r/(k-\alpha)}$ red copies of $H$ on those
  vertices, as any such copy must use at least $k-\alpha$ vertices of $R$
  (recall, $B$ is fully blue). Similarly, there are at most $\floor{r/k}$ blue
  copies of $H$. We aim to find this many red and blue copies of $H$ in $R_5
  \cup A_r \cup B_5 \cup A_b \cup T$ in our original colouring $c$ as this
  would then complete the proof. In fact, since we can find exactly $t$ copies
  of $H$ in any colour in $T$, our task is reduced to finding $\floor{|R_5 \cup
  A_r|/(k-\alpha)}$ red copies and $\floor{|B_5 \cup A_b|/k}$ blue copies in
  $R_5 \cup A_r \cup B_5 \cup A_b$.

  Since $B_5 \subseteq B_3 \subseteq B_0$ and $A_b$ is an $(H,r(H))$-absorber
  for $B_0$ in blue, we can find a $H$-tiling of $B_5 \cup A_b$: any set of
  $r(H)$ vertices in $B_5$ contains a blue $H$ (as $B_0$ and thus $B_5$ is red
  $H$-free) with the remaining fewer than $r(H)$ vertices then getting absorbed
  into an $H$-tiling by $A_b$. As $A_r$ is an $(\hat H,r(H,\hat H))$-absorber
  for $R_5 \subseteq R_3$ in red, we can similarly find an $\hat H$-tiling of
  $R_5 \cup A_r$. What remains to be done is to, for each such red $\hat H$,
  find $\alpha$ common neighbours in $B_5$ to complete a red $H$. This is easily
  achieved, as by Claim~\ref{cl:red-degree-across} there are at most
  $(k-\alpha)r(H)$ vertices in $B_5$ with blue edges to some $k-\alpha$ vertices
  of $R_5 \cup A_r \subseteq R_3$, and at any point we used up at most
  \[
    \frac{|R_5|}{k-\alpha} \cdot \alpha \leq \frac{|R_3|}{k-\alpha} \cdot \alpha
    = \frac{\alpha}{k} \cdot |B_3|.
  \]
  As (using $\alpha \leq (1-\frac{1}{\Delta+1})k$ once again),
  \[
    |B_5| \geq |B_3| - o(km) - k|E| \geq |B_3| - o(km) \geq \frac{\alpha}{k}
    |B_3| + (k-\alpha)r(H)
  \]
  there is a desired collection of $\floor{|R_5|/(k-\alpha)}$ red copies of $H$,
  finally completing the proof.
\end{proof}

Note that from the proof we can recover even better bounds on the order of sets
$|R|, |B|$, depending on the colour of the `monochromatic edges' in $[R,B]$.
Under the assumption that the edges in $[R,B]$ are red (otherwise, everything is
the same mutatis mutandis),
\[
  |B| = \frac{k}{2k-\alpha}\Big(N - |E| - |A_r| - |A_b|\Big) \geq nk - o(nk)
  \geq \alpha n + k(|E|+1).
\]
That said, in exactly the same way as in \cite{bucic2021tight} (for that matter,
also very similar to \cite{burr1987ramsey})
Theorem~\ref{thm:structure-in-critical-colouring} immediately implies
Theorem~\ref{thm:main-symmetric}. We omit the proof.

Tightness in dependence of $n$ on $k = |H|$ is the highlight of
Theorem~\ref{thm:main-symmetric}, but in order to discuss it we unfortunately
need additional notation. Recall, $\cD(H)$ is defined as a collection of graphs
obtained by removing any maximal independent set of $H$. Let $\cD'(H) \subseteq
\cD(H)$ be those obtained by removing any maximum independent set. Now, let us
define $\cD_c$ and $\cD_c'$ as the collection of all the connected components of
the elements of $\cD(H)$ and $\cD'(H)$, respectively. The following assertion
follows analogously as \cite[Theorem~13]{bucic2021tight} and is a corollary of
Theorem~\ref{thm:structure-in-critical-colouring}.

\begin{proposition}\label{prop:tight-bounds}
  Let $H$ be a connected $k$-vertex graph with maximum degree $\Delta$. Then,
  provided $n \geq 2^{O(\Delta\log^2\Delta)}k$,
  \[
    r(\cD_c(H), \cD(H)) - 2 \leq r(nH) - (2k-\alpha)n \leq r(\cD_c'(H),
    \cD'(H)) - 2.
  \]
\end{proposition}

The above statement is very useful in a way that if the Ramsey numbers in the
upper and lower bound are the same then the proposition gives an exact formula
for $r(nH)$ (giving a precise constant $c = c(H)$ in
Theorem~\ref{thm:main-symmetric}). In fact, this is often the case, and examples
for which the bounds are not tight are difficult to find (for more details
see~\cite{bucic2021tight}). As mentioned by Burr~\cite{burr1987ramsey}, an
important class of graphs $H$ for which the bounds coincide are the ones that
have a vertex $v$ whose $N_H(v)$ is contained in some maximum independent set of
$H$ (e.g.\ all bipartite graphs and cycles of length at least four). For any
such $H$, the collections $\cD_c(H)$ and $\cD_c'(H)$ contain an isolated vertex,
and thus both Ramsey numbers above evaluate to one! In other words $r(nH) =
(2k-\alpha)n - 1$ as soon as $n$ is large enough for the long term behaviour to
kick in.

Let $C_k^2$ be the square\footnote{A square $H^2$ of a graph $H$ is obtained by
adding an edge between every two vertices at distance at most two in $H$.} of a
cycle on $k = 3\ell + 4$ vertices for some integral $\ell \geq 4$. Fix any
vertex, remove all edges with both endpoints in its neighbourhood, and call the
obtained graph $H_k$. Clearly, $\Delta(H_k) = 4$, and $\alpha(H_k) = \ell+3$
(see Figure~\ref{fig:square-cycle}).

\begin{figure}[!htbp]
  \centering
  \includegraphics[scale=1]{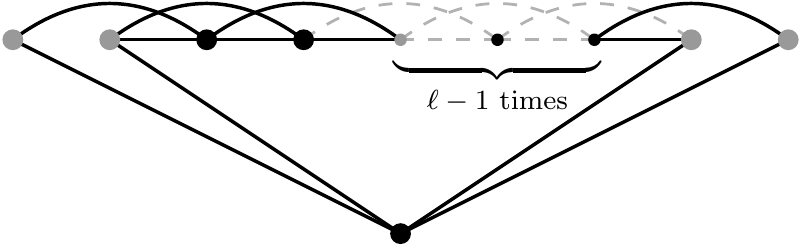}
  \caption{The graph $H_k$ obtained from $C_k^2$ with the dashed part repeated
  $\ell-1$ times and its maximum independent set (greyed out vertices).}
  \label{fig:square-cycle}
\end{figure}

We show below that if $n < \ell/2 = O(k)$, then $r(nH_k) \geq
(2k-\alpha(H_k))n$. On the other hand, our graph has a vertex whose
neighbourhood belongs to a maximum independent set, and thus $r(nH_k) =
(2k-\alpha(H_k))n-1$ by Proposition~\ref{prop:tight-bounds} as soon as $n$ is
`large enough'. This implies that in order for \eqref{eq:long-term-behaviour} to
hold for bounded degree graphs, it is necessary for $n$ to have a linear
dependence on $k$, as in Theorem~\ref{thm:main-symmetric}.

\begin{proposition}\label{prop1}
  For $n \leq \ell/2 - 1$ we have $r(nH_k) \geq (2k-\alpha(H_k))n$.
\end{proposition}
\begin{proof}
  Let $R$ be a red clique of order $(k-\alpha-2)n-1$, $B$ a blue clique of order
  $kn-1$, and $E$ a blue clique of order $\ell$. Colour all edges in $[E,R]$
  blue, $[E,B]$ red, and $[R,B]$ red. Note that the obtained complete graph $K$
  has
  \[
    (k-\alpha-2)n - 1 + kn - 1 + \ell = (2k-\alpha)n + \ell - 2(n+1) \geq
    (2k-\alpha)n - 1,
  \]
  vertices. It is obvious that $K$ has no blue $nH_k$ as $B$ can hold at most
  $n-1$ blue copies of $H_k$, and every blue copy intersecting $E$ must do so in
  at least $k-\alpha > \ell$ vertices. On the other hand, $K$ cannot have a red
  $nH_k$ either. Indeed, every red copy of $H_k$ can intersect $B$ only in an
  independent set and $E$ only in some isolated vertices obtained after
  embedding an independent set into $B$. Since almost all vertices have an edge
  in their neighbourhood, it is easy to see that the largest such collection of
  independent vertices is of order two (obtained by removing a maximal
  independent set including the special vertex). Thus, every red $H_k$ has to
  intersect $R$ in at least $k-\alpha-2$ vertices. Consequently, there are at
  most $kn-1$ of those.
\end{proof}

\subsection{Graphs with a given number of edges}

In a similar fashion to Theorem~\ref{thm:structure-in-critical-colouring}, if
$H$ is a graph with $m$ edges, we can provide a bound on $n$ that depends only
on $m$. The main ingredient is again a lemma for existence of $H$-ties, which is
a bit more involved this time around.

\begin{lemma}\label{lem:H-tie-lemma-edges}
  Let $H$ be a graph with $m$ edges and no isolated vertices. Given a
  $2$-coloured complete graph $K$ on vertices $R \cup B$ such that $K[R]$ has no
  blue $H$, $K[B]$ has no red $H$, and $|R|, |B| \geq 2^{O(\sqrt{m}\log^2m)}$,
  there exists an $H$-tie with red part in $R$ and blue part in $B$.
\end{lemma}
\begin{proof}
  Let $k := |H|$. First, let us establish a simple property that comes in handy
  almost immediately. Let $R' \subseteq R$ and $B' \subseteq B$ be sets of order
  at least $2^{O(\sqrt{m}\log^2 m)} \gg r(H)$. Observe that $K[R']$ contains at
  most $|R'|/2$ vertices with blue degree $(1-\frac{1}{2k})|R'|$ (as otherwise
  by Tur\'{a}n's theorem we would have a blue $H$). Hence, at least half of the
  vertices have red degree at least $|R'|/(2k)$. Analogously, at least half of
  vertices in $B'$ have blue degree at least $|B'|/(2k)$. Let $R''$ and $B''$ be
  sets of these high degree vertices. Now, for either of the colours, there has
  to exist $v \in R''$ that has at least $|B''|/(k+1) \geq |B'|/(2k+2)$
  neighbours in $B''$ in that colour. Indeed, if there is no such vertex, there
  is a red $H$ in $R''$ (as $|R''| \geq r(H)$) whose vertices have at least
  $|B''| - k|B''|/(k+1) \geq r(H)$ common neighbours in $B''$ in the other
  colour, forming an $H$-tie. Similarly, for either colour there has to exist $v
  \in B''$ with at least $|R'|/(2k+2)$ neighbours in $R''$ that colour.

  We now proceed with the proof. Our goal here is to find a red clique of order
  $\sqrt{m}/2$ in $R$ with a large common red neighbourhood in $R$ and common
  blue neighbourhood in $B$. We also find a blue clique of order $\sqrt{m}/2$ in
  $B$ with similar properties with colours reversed. The whole process is then
  iterated within the found common neighbourhoods to additionally construct two
  slightly smaller cliques, both of order $\sqrt{m}/4$. This time around, the
  cliques both have large common neighbourhoods in their respective colours
  (both red for the red clique and both blue for the blue one). All four cliques
  then serve to accommodate $\sqrt{m}/2$ vertices of large degree in $H$ when
  looking for an $H$-tie, while the rest is to be found via
  Lemma~\ref{lem:H-tie-lemma}. For a more visual explanation of the setup we aim
  for, we refer to Figure~\ref{fig:clique-gadget} below.

  \begin{figure}[!htbp]
    \centering
    \includegraphics[scale=0.6]{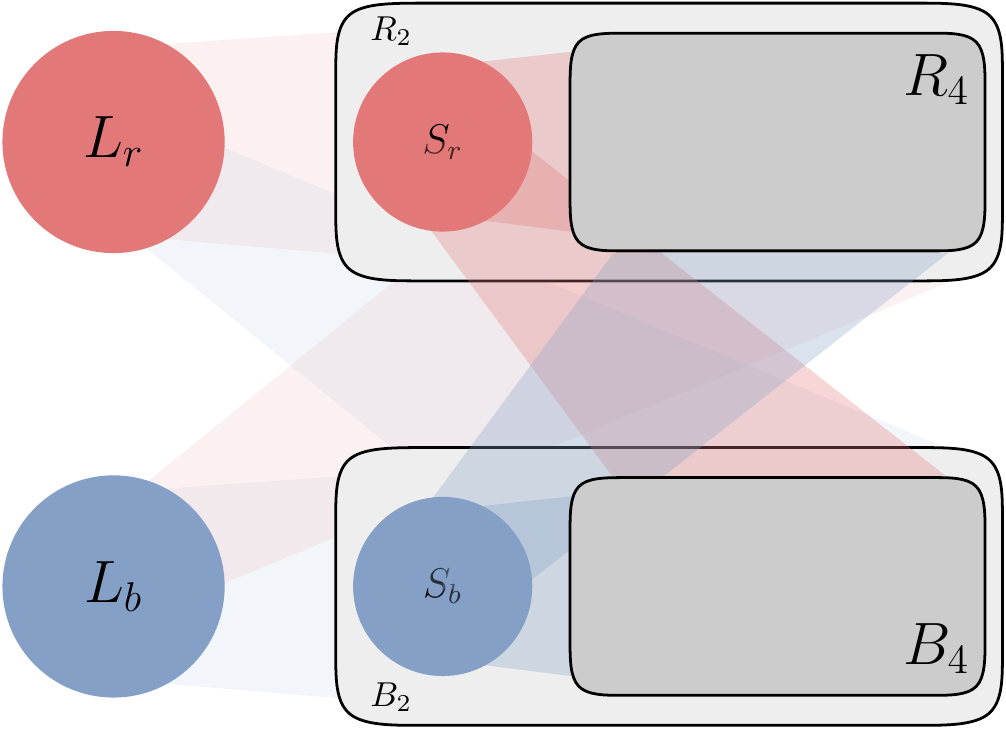}
    \caption{The setup obtained after the procedure: $L_r$ and $L_b$ are cliques
      of order $\sqrt{m}/2$ in their respective colours and $S_r$ and $S_b$ are
      cliques of order $\sqrt{m}/4$ in their respective colours; all edges in
      $[L_r,R_2]$, $[L_b,R_2]$, $[S_r,R_4]$, $[S_r,B_4]$ are red and
      $[L_r,B_2]$, $[L_b,B_2]$, $[S_b,R_4]$, $[S_b,B_4]$ are blue We plan to
      find a part of an $H$-tie in $[R_4,B_4]$ and use the cliques appropriately
      to extend it to a full $H$-tie. The edges in $[L_r,L_b]$ and $[S_r,S_b]$
    do not matter.}
    \label{fig:clique-gadget}
  \end{figure}

  The outlined aim is accomplished iteratively in the usual spirit of
  Ramsey-type arguments. Initially, set $R_0 := R$ and $B_0 := B$. Take a vertex
  $v$ in $R_0$ as from the property above, namely $v$ has at least
  $|R_0|/(2k+2)$ red neighbours in $R_0$ and at least $|B_0|/(2k+2)$ blue
  neighbours in $B_0$. Update the sets $R_0$ and $B_0$ to now denote these
  neighbourhoods. This process repeats as long as $|R_0|, |B_0| \geq
  2^{O(\sqrt{m}\log^2 m)}$, and so at least $\sqrt{m}/2$ times, constructing the
  desired red clique $L_r$ (stands for `large red').

  Let $R_1$ and $B_1$ be the common neighbourhoods in the colours as above of
  $L_r$ in $R$ and $B$, respectively. These are now of order
  \[
    |R_1| \geq |R|/(2k+2)^{\sqrt{m}/2} \geq |R| \cdot 2^{-\Omega(\sqrt{m}\log
    k)} \geq 2^{O(\sqrt m \log^2 m)},
  \]
  and similarly $|B_1| \geq 2^{O(\sqrt m\log^2 m)}$. Repeat the same process to
  find a blue clique $L_b$ in $B_1$ by starting with the newly obtained sets
  $R_1$ and $B_1$ (this time the edges across are red, see
  Figure~\ref{fig:clique-gadget}). Denote its common neighbourhoods in $R_1$ and
  $B_1$ by $R_2$ and $B_2$, respectively.

  In an analogous way, and starting with $R_2$ and $B_2$, we repeat the process
  of first finding a red clique $S_r$ (stands for `small red') in $R_2$ of order
  $\sqrt{m}/4$, with its now red common neighbourhoods $R_3 \subseteq R_2$ and
  $B_3 \subseteq B_2$, and subsequently a blue clique $S_b$ in $B_3$ of order
  $\sqrt{m}/4$ with its blue common neighbourhoods $R_4 \subseteq R_3$ and $B_4
  \subseteq B_3$ (see Figure~\ref{fig:clique-gadget}). Again, both $R_4$ and
  $B_4$ are large enough.

  Fix a largest independent set $I$ of $H$, and obtain a graph $H'$ by removing
  $\sqrt{m}/4$ vertices of largest degree from $H - I$ and $I$ each. This graph
  has maximum degree $8\sqrt{m}$. Let $I' \subseteq I$ be the remaining vertices
  which still form an independent set of $H'$. As $|R_4|, |B_4| \geq 2^{256
  \cdot 8\sqrt{m}\log^2(8\sqrt{m})}$, by Lemma~\ref{lem:H-tie-lemma} there
  exists\footnote{We cannot guarantee which independent set is used in the
  lemma, but the proof is exactly the same if we fix an independent set
  beforehand.} an $H'$-tie with, say, red edges across. Such an $H'$-tie then
  consists of a red $H'-I'$ in $R_4$ and blue $H'$ in $B_4$. To obtain an
  $H$-tie with red edges across, we extend the red $H'-I'$ by $S_r$ (consists of
  exactly $\sqrt{m}/4$ vertices) and, noting that all the edges in $[L_b,R_4]$
  are red, the blue $H'$ by $L_b$ (consists of exactly $\sqrt{m}/2$ vertices).
  Hence, the number of vertices used is exactly $2|H'| - (\alpha(H) -
  \sqrt{m}/4) + 3\sqrt{m}/4 = 2|H|-\alpha(H)$. The blue copy of $H'$ in $B_4$
  together with $L_b$ contain a blue $H$ as a subgraph, and similarly the red
  copy of $H'-I'$ in $R_4$ together with $S_r$ contain a red $H-I$ as a
  subgraph. The edges across are red by construction.
\end{proof}

In order to show that there is a critical colouring for $nH$ with a good
structure, we proceed as in the proof of
Theorem~\ref{thm:structure-in-critical-colouring}. However, there are three
small differences: one is Lemma~\ref{lem:H-tie-lemma-edges} as mentioned,
another is utilising the general absorbing lemma
(Lemma~\ref{lem:absorbing-lemma-general}) instead of the one tailored for
bounded degree graphs, and lastly, we extensively made use of the fact that
$\alpha(H) \leq \Delta(H)k/(\Delta(H)+1)$, which is now only $\alpha(H) \leq
k-1$. However, noting that the latter two only incur an additional polynomial
factor of $k$ (loosely, around $k^{17}$) and as $H$ has no isolated vertices $m
\geq k/2$, these do not matter compared to the exponential
$2^{O(\sqrt{m}\log^2m)}$.

\begin{theorem}\label{thm:structure-in-critical-colouring-edges}
  Let $H$ be a graph with $m$ edges and no isolated vertices. Then, provided $n
  \geq 2^{O(\sqrt{m}\log^2 m)}$, there exists a $2$-colouring of $K_{r(nH)-1}$
  with no monochromatic $nH$ and the following structure: we can partition the
  vertex set into $R$, $B$, and $E$, such that:
  \begin{itemize}
    \item $|E| \leq r(H)$ and $|R|, |B| \geq k(|E|+1)$;
    \item all edges inside $R$ are red and all edges inside $B$ are blue;
    \item all edges between $R$ and $B$ are of the same colour, all edges
      between $R$ and $E$ are blue and all edges between $B$ and $E$ are red;
    \item there is no $H$-tie containing a vertex of $E$.
  \end{itemize}
\end{theorem}

This establishes Theorem~\ref{thm:main-symmetric-edges} in the same way as
Theorem~\ref{thm:structure-in-critical-colouring} is used to prove
Theorem~\ref{thm:main-symmetric}.

%% file: concluding.tex
\section{Concluding remarks}\label{sec:concluding-remarks}

In this paper, we have proven several bounds for Ramsey number for the
collection of $n$ copies of a graph $H$, in various specific cases of general
interest. The main restriction we were concerned with is one of the most natural
and well-studied ones in Ramsey theory---when $H$ has maximum degree bounded by
a fixed integer $\Delta$. Looking closely at our methods, one can deduce that
the only property we require of the host graph is that it contains a large
subgraph in which any two disjoint sets of linear order induce a subgraph of
large density. This is encapsulated through the notion of
bi-$(\eps,\gamma)$-dense graphs, and implemented via,
Lemma~\ref{lem:dense-to-bi-dense}, Lemma~\ref{lem:greedy-embedding}, and
Lemma~\ref{lem:embedding-lemma}. In fact, this is the main reason why both in
the absorbing lemma (Lemma~\ref{lem:efficient-absorbers}) and the $H$-tie lemma
(Lemma~\ref{lem:H-tie-lemma}), we require $N$ to be of the order
$2^{O(\Delta\log^2\Delta)}k$.

However, if we additionally restrict $H$ to be bipartite, then we can do
somewhat better. The following result has been obtained by Fox and
Sudakov~\cite{fox2009density} and relies on a modified version of dependent
random choice (Lemma~\ref{lem:dependent-random-choice}).

\begin{theorem}[Theorem~2.3 in \cite{fox2009density}]\label{thm:bipartite-ramsey}
  Let $k \in \N$ and let $H$ be a bipartite graph on $k$ vertices with maximum
  degree $\Delta$. For $\eps > 0$, if $G = (V_1, V_2; E)$ is a bipartite graph
  with $|V_1| = |V_2| = N \geq 16\Delta\eps^{-\Delta} k$ and at least $\eps N^2$
  edges, then $H$ is a subgraph of $G$.
\end{theorem}

Using this, we can in the same way prove a bipartite analogue of
Lemma~\ref{lem:efficient-absorbers} and Lemma~\ref{lem:H-tie-lemma}.

\begin{lemma}
  Let $G$ and $H$ be two bipartite graphs with maximum degree $\Delta$ and $k =
  \max\{|G|,|H|\}$. Let $K$ a graph of order $N \geq 2^{O(\Delta)}k$ whose
  complement is $G$-free. Then, there exists $X \subseteq V(K)$ of order at
  least $N/2^{O(\Delta)}$ which is
  bi-$(2^{-4\Delta-11}\Delta^{-4},2^{-1})$-dense and an $(H,r(G,H))$-absorber
  for $X$ in $K$ of order at most $2^{O(\Delta)}k$.
\end{lemma}
\begin{proof}[Sketch of the proof]
  The only thing we need to do differently than in
  Lemma~\ref{lem:efficient-absorbers} is how we find a large
  `bi-$(\eps,2\gamma)$-dense' subgraph of $K$, for appropriately chosen
  parameters. Let $\gamma = 1/4$ and $\eps = \gamma^{2\Delta}/(2^{11}\Delta^4)$.
  Consider any disjoint $X, Y \subseteq V(K)$ of order $\eps N \geq 16\Delta
  2^\Delta k$. Since the complement of $K$ is $G$-free, the density in the
  complement of $K[X,Y]$ can be at most $1-2\gamma$ by
  Theorem~\ref{thm:bipartite-ramsey} applied with $1-2\gamma$ (as $\eps$), and
  thus the density in $K[X,Y]$ is at least $2\gamma$.
\end{proof}

However, to prove an analogue of Lemma~\ref{lem:H-tie-lemma} we require $N \geq
2^{O(\Delta\log\Delta)}k$. The crux comes in the case when $[R,B]$ is \emph{not}
bi-$(\eps,\gamma)$-dense in either of the colours. There, e.g.\ we need every
set of $\Delta$ vertices of $R$ to have a somewhat large common neighbourhood in
$B$, which is established through minimum degree of the form $(1-\gamma)|B|$,
say. This forces $\gamma$ to be of the order $1/\Delta$ and propagates into
applying Lemma~\ref{lem:embedding-lemma} to $R$ (or $B$): $\gamma^\Delta|R|$
must be larger than both $|R|$ and $k$. In conclusion, both $R$ and $B$ need to
be of the order $2^{O(\Delta\log\Delta)}k$. This lemma is an indispensable part
of our argument, and we did not find a way around it. (One would encounter a
similar issue when trying to prove the asymmetric version, an analogue of
Theorem~\ref{thm:main-asymmteric-bounded-deg} for bipartite graphs.)

\begin{lemma}
  Let $H$ be a $k$-vertex bipartite graph with maximum degree $\Delta$. Given a
  $2$-coloured complete graph $K$ on vertices $R \cup B$ such that $K[R]$ has no
  blue $H$, $K[B]$ has no red $H$, and $|R|, |B| \geq 2^{O(\Delta\log\Delta)}k$,
  there exists an $H$-tie with red part in $R$ and blue part in $B$.
\end{lemma}

That said, the following is a result we can obtain when $H$ is a $k$-vertex
bipartite graph of maximum degree $\Delta$. Note that, similarly as before, the
requirement for $n$ is (almost) optimal up to the $\log\Delta$ factor in the
exponent, as $r(H) = 2^{O(\Delta)}k$ (see \cite{graham2000graphs} for the lower
bound and \cite{conlon2008new, fox2009density} for the upper bound). However, it
is unclear whether a dependence in $k$ is necessary, as no construction similar
to the one in Proposition~\ref{prop:symmetric-lower-bound} would work in the
bipartite case.

\begin{theorem}\label{thm:main-symmetric-bipartite}
  Let $H$ be a $k$-vertex bipartite graph with no isolated vertices and maximum
  degree $\Delta$. There is a constant $c = c(H)$ and $n_0 =
  2^{O(\Delta\log\Delta)}k$ such that for all $n \geq n_0$
  \[
    r(nH) = (2k - \alpha(H))n + c.
  \]
\end{theorem}

All this leaves an obvious open problem of trying to remove the logarithmic
factor in the exponent for several classes of graphs $H$: (bipartite) bounded
degree graphs, and graphs with a fixed edge number. Furthermore, the general
natural question of Buci\'{c} and Sudakov~\cite{bucic2021tight}, showing that $n
= O(r(H)/|H|)$ is sufficient to establish a long term behaviour on the Ramsey
number of $nH$, remains open and intriguing.

There are at least two other notions of sparsity that are studied in the
literature with respect to Ramsey numbers. The most notable one is that of
$d$-degenerate graphs (see, e.g., \cite{conlon2015recent, fox2009two,
lee2017ramsey}). It is plausible that a conditional improvement, similar to the
one we obtained for bounded degree graphs, is within reach for $r(nH)$ when $H$
is $d$-degenerate. There are several obstacles to be overcome if one is to
follow the strategy outlined here, and we did not pursue this direction further.
Another possibility and a step towards it would be to look at an intermediate
class between bounded degree and bounded degeneracy, that of $p$-arrangeable
graphs. We refer to the survey~\cite{conlon2015recent} for definitions and known
results in this scenario, but do not delve into it more.